\documentclass{article}

\usepackage{fullpage}
\usepackage{graphicx} 
\usepackage{amsmath,amssymb,amsthm}
\usepackage{mathrsfs}
\usepackage{xcolor}
\usepackage{indentfirst}
\usepackage{hyperref}
\usepackage{booktabs}
\usepackage{float}
\usepackage{appendix}
\usepackage{titlesec}
\usepackage{chngcntr}

\newtheorem{theorem}{Theorem}
\newtheorem{lemma}{Lemma}
\newtheorem{proposition}{Proposition}
\newtheorem{definition}{Definition}
\newtheorem{remark}{Remark}[section]
\newtheorem{corollary}{Corollary}

\numberwithin{equation}{section}

\newcommand{\vertiii}[1]{{\left\vert\kern-0.25ex\left\vert\kern-0.25ex\left\vert #1 
    \right\vert\kern-0.25ex\right\vert\kern-0.25ex\right\vert}}
\DeclareMathOperator{\codim}{\text{codim}}

\title{Existence of solutions to Dirichlet boundary value problems of the stationary relativistic Boltzmann equation}

\author{Yi Wang\thanks{School of Mathematics and Statistics, Ningbo University, Ningbo, China.} \and Li Li$^{*,}$\footnote{Corresponding author: \texttt{lili2@nbu.edu.cn}} \and Zaihong Jiang\thanks{School of Mathematical Sciences, Zhejiang Normal University, Jinhua, China}}
\date{\today}

\begin{document}

\maketitle

\begin{abstract}
    In this paper, we study the Dirichlet boundary value problem of steady-state relativistic Boltzmann equation in half-line with hard potential model, given the data for the outgoing particles at the boundary and a relativistic global Maxwellian with nonzero macroscopic velocities at the far field. We first explicitly address the sound speed for the relativistic Maxwellian in the far field, according to the eigenvalues of an operator based on macro-micro decomposition. Then we demonstrate that the solvability of the problem varies with the Mach number $\mathcal{M}^\infty$. If $\mathcal{M}^\infty<-1$, a unique solution exists connecting the Dirichlet data and the far field Maxwellian when the boundary data is sufficiently close to the Maxwellian. If $\mathcal{M}^\infty>-1$, such a solution exists only if the outgoing boundary data is small and satisfies certain solvability conditions. The proof is based on the macro-micro decomposition of solutions combined with an artificial damping term. A singular in velocity (at $p_1=0$ and $|p|\gg 1$) and spatially exponential decay weight is chosen to carry out the energy estimates. The result extends the previous work [Ukai, Yang, Yu, Comm. Math. Phys. 236, 373-393, 2003] to the relativistic problem. 
    
\end{abstract}

\section{Introduction}

The relativistic Boltzmann equation reveals the statistical behavior of high-speed particles, which is crucial for exploring relativistic fluid dynamics and studying the dynamical behavior of particles in extreme physical environments. The relativistic Boltzmann equation is
\begin{equation}\label{eq:RBE}
    \partial_t F + \hat{p} \cdot \nabla_x F = Q(F,F),
\end{equation}
where $F=F(t, x, p)$ represents a distribution function for particles at time $t > 0$, position $x \in \Omega$ with velocity $p \in \mathbb{R}^3$. Let $\mathfrak{c}$ denote the speed of light. The energy of a relativistic particle with velocity $p$ is defined by $p_0:=\sqrt{\mathfrak{c}^2+|p|^2}$ and the normalized particle velocity $\hat{p}$ is given by
\[
    \hat{p} := \mathfrak{c} \frac{p}{p_0} \equiv \frac{p}{\sqrt{1+\frac{|p|^2}{\mathfrak{c}^2}}}.
\]
The relativistic Boltzmann collision operator $Q(F, F)$ is given, for example in \cite{Boisseau:vanLeeuwen,deGroot:vanLeeuwen:vanWeert}, by 
\begin{equation}\label{eq:Q}
    Q( F_1, F_2 ) = \frac{\mathfrak{c}}{2} \frac{1}{p_0} \int_{\mathbb{R}^3} \int_{\mathbb{R}^3} \int_{\mathbb{R}^3} W( p, q \mid p',q' ) \big( F_1 ( p' ) F_2 ( q' ) - F_1( p ) F_2 ( q ) \big) \frac{dp'}{p_{0}'} \frac{dq'}{q_{0}'} \frac{dq}{q_0}. 
\end{equation}
The translation rate $W(p,q\mid p',q')$ can be expressed in the following form:
\[
    W ( p, q \mid p',q' )=s\sigma ( g, \theta ) \delta ^{( 4 )}( P+Q-P'-Q' ),
\]
where $\sigma ( g,\theta )$ is the scattering kernel measuring the interactions between particles, and Dirac function $\delta ^{( 4 )}$ is the delta function of four variables. 
The relativistic velocities are represented by four-vectors $P, Q, P'$ and $Q'$, where $p, q\in \mathbb{R}^3$ are the pre-collisional velocity, $(p_0,q_0)$ are the pre-collisional energy, $P$ and $Q$ are given by 
\[
    P := 
    \begin{pmatrix}
		p_0 \\ p
	\end{pmatrix} = 
    \begin{pmatrix}
		\sqrt{\mathfrak{c}^2+|p|^2} \\ p
	\end{pmatrix}, \qquad 
    Q := 
    \begin{pmatrix}
		q_0 \\ q
	\end{pmatrix} = 
    \begin{pmatrix}
		\sqrt{\mathfrak{c}^2+|q|^2} \\ q
	\end{pmatrix}. 
\]
The post-collisional velocity $p',q'\in \mathbb{R}^3$ and energy $(p_0',q_0')$ satisfy:
\begin{equation}\label{eq:conservation:laws}
    P + Q = P' + Q', \qquad P' = (p'_0,p')^T, \qquad Q' = (q'_0,q')^T. 
\end{equation}
We will explain $s,g$ and $\theta$ in section \ref{sub:notations}.

Four common boundary conditions are the Dirichlet (perfect absorption), specular reflection, reverse reflection, and diffuse reflection boundary conditions. 
The boundary value problem of the classical Boltzmann equation (non-relativistic Boltzmann equation) has been the subject of extensive study by numerous researchers, resulting in a wealth of significant findings.
The studies of the existence of solutions to nonlinear boundary value problems of the Boltzmann equation with Dirichlet boundary condition under the hard sphere model, hard potential model, and soft potential model are presented in \cite{Ukai:Yang:Yu}, \cite{Chen:Liu:Yang}, \cite{Wang:Yang:Yang} respectively.
Later, Tian studied the existence of solutions to the Boltzmann equation under the hard sphere model affected by reverse, diffuse and specular reflection boundary conditions in \cite{Tian:reverse} and \cite{Tian:physical}, respectively. 
Similar to the result in \cite{Ukai:Yang:Yu}, Tian proved that the existence of a solution to the Dirichlet boundary value problem depends on the Mach number of the far field Maxwellian. However, for the case of reverse and specular reflection boundary conditions, the method is only effective for positive Mach numbers. 
The existence of solution to the Boltzmann equation with angular cutoff was proved for hard potential model in \cite{Sun:Tian:hardpotential} and for soft potential model in 
\cite{Sun:Tian:softpotential} under the mixed boundary conditions, which is a linear combination of the Dirichlet boundary condition and weak diffuse reflection boundary condition at the wall surface. 
On the basis of the existence of the solution, the studies on its stability can be referred to \cite{Tian:Sun:stability:diffusive,Ukai:Yang:Yu:stability,Wang:Yang:Yang:stability:hard,Yang:Zhao:stability:specular,Yang:stability:soft}.

The relativistic Boltzmann equation is an important equation in relativistic kinetic theory. In \cite{Marek:Maria}, the study focused on the linearized relativistic Boltzmann equation in \( L^2(\mathbb{R}, p) \) and demonstrated that the collision operator can be expressed as a multiplication operator plus a compact perturbation, while also establishing the existence and uniqueness of solutions. 
It was proved in \cite{Glassey:Strauss} that, under appropriate conditions on the scattering kernel, there exists a unique global solution when initial data is periodic in space variable and close to equilibrium. Moreover, the relativistic Maxwellian distribution is asymptotically stable. 
In \cite{Glassey:Cauchy}, the study focused on the Cauchy problem for the relativistic Boltzmann equation with small data, and global "mild" solutions are obtained for an appropriate class of scattering cross sections. 
The local-in-time hydrodynamic limit of the relativistic Boltzmann equation was established in \cite{Speck:Strain}. 
The global existence and uniqueness of solution to the relativistic Boltzmann equation with soft potentials on $\mathbb{T}^3_x$, which decays with any polynomial rate towards a steady-state relativistic Maxwellian, was proved in \cite{Strain:stability}.
Later, Duan and Yu proved in \cite{Duan:Yu} the existence of global-in-time solutions that approach to equilibrium states at an exponential decay rate, while extending previous results on the classical Boltzmann equation in \cite{Caflisch:I:soft,Caflisch:II:soft,Strain:Guo:Exponential} and improving the result on almost exponential time-decay in \cite{Strain:stability}. 
The global existence, uniqueness, positivity and optimal time convergence rate of solution to the relativistic Boltzmann equation with small perturbed initial data in the whole space $\mathbb{R}^3_x$ was addressed in \cite{Strain:Zhu}, based on a general soft potential assumption proposed in \cite{Marek:Maria}. 
The study in \cite{Jang} and \cite{Jang:Strain} established the global-in-time existence, uniqueness and asymptotic stability of solutions to the relativistic Boltzmann equation nearby the relativistic Maxwellian without angular cutoff, and resolved an open question in relativistic kinetic theory. 
The global existence and uniqueness of mild solution to the relativistic Boltzmann equation with large amplitude initial data in both whole space and torus was proved in \cite{Wang}, under some additional smallness conditions on $L_x^1 L_p^\infty$-norm, as well as defect mass, energy and entropy. 
A longstanding question regarding the Jacobian determinant for the relativistic Boltzmann equation in the center-of-momentum coordinates was addressed in \cite{Chapman:Jang:Strain}. 
More background descriptions on the relativistic Kinetic theory can refer to \cite{Cercignani:Kremer,deGroot:vanLeeuwen:vanWeert,Glassey,Stewart}. 

Studying the boundary value problem of the relativistic Boltzmann equation can help reveal the behavior of particles near the boundary and provide guidance for engineering applications. In this paper, 
we consider the following steady-state problem in half line with Dirichlet boundary condition under the hard potential model:
\begin{equation}\label{eq:BVP}
    \left\{ 
    \begin{array}{ll}
        \hat{p}_1 \partial_x F = Q( F, F ), \qquad & x\in \mathbb{R}^+, \quad p\in \mathbb{R}^3,\\
        F|_{x=0} = F_0( p ), & p_1>0, \quad ( p_2, p_3 ) \in \mathbb{R}^2,\\
        \displaystyle \lim_{x\to +\infty} F = J_{\infty}( p ),  & p\in \mathbb{R}^3,
    \end{array} \right.
\end{equation}
where $F(x,p)$ is a distribution function for fast moving particles at position $x\in \mathbb{R}^+$ with velocity $p \in \mathbb{R}^3$, $F_0$ is the boundary condition assigned for the outgoing particles at $x=0$, while $J_\infty$ is the far field relativistic Maxwellian distribution. 

We consider the dimensionless hard potential model, see \cite{Glassey:Strauss} for example, with $\sigma$ satisfying
\begin{equation}\label{eq:sigma}
    c_1 \frac{(g/\mathfrak{c})^{a+1}}{\sqrt{s}/\mathfrak{c}} \sin^{\varsigma} \theta < \sigma(g,\theta) <  c_2 \big( (g/\mathfrak{c})^a + (g/\mathfrak{c})^{-b} \big) \sin^{\varsigma} \theta,
\end{equation}
where $c_1$, $c_2$ are positive constants, 
\begin{equation}\label{eq:ab}
    0 \le b < \frac{1}{2}, \quad 0 \le a < 2-2b, 
\end{equation}
and either $\varsigma \ge 0$ or
\begin{equation}\label{eq:gamma}
    | \varsigma | < \min \Big\{ 2-a, \frac{1}{2} - b, \frac{1}{3}(2-2b-a) \Big\}. 
\end{equation}

The steady solutions of \eqref{eq:RBE} are the well-known J\"{u}ttner solutions, also known as the relativistic Maxwellian states, as referenced for example in \cite{Speck:Strain}. Given any functions $\rho(t,x)$, $T(t,x)$ and $U(t,x)=(u_0,u)$ with $\rho>0$, $T>0$, $u\in\mathbb{R}^3$, $u_0 := \sqrt{\mathfrak{c}^2 + |u|^2}$, we define the corresponding relativistic Maxwellian state $J_{[\rho,u,T]}$ as follows:
\[
    J_{[\rho, u, T]} ( p ) := \frac{\rho}{ 4 \pi m^2 k_B \mathfrak{c} T K_2( \frac{m \mathfrak{c}^2}{k_B T} ) } e^{-\frac{P\cdot U}{k_B T} },
\]
where $m$ is the rest mass of particles, $k_B$ is the Boltzmann constant and $K_\alpha(z)$ is the modified Bessel function defined by
\[
    K_\alpha(x) := \int_{0}^{\infty} e^{-x \cosh t }\cosh\alpha t dt. 
\]

Without loss of generality, we assume that $k_B = m = 1$. Define a dimensionless variable $z$ by
\[
    z:=\frac{\mathfrak{c}^2}{T}.
\]
Then 
\[
    J_{[ \rho, u,T ]}(p) :=\frac{\rho z}{4\pi \mathfrak{c}^3 K_2(z)} e^{ - \frac{z P \cdot U}{\mathfrak{c}^2} }.
\]
Note that $(P\cdot U)/ \mathfrak{c}^2$ is dimensionless. The expression of $J$ is meaningful in the sense that $\int J dp$ has the same unit as density $\rho$.

Let $\rho_\infty=1$. We take the far field distribution $J_\infty$ as
\begin{equation}\label{eq:J:infty}
    J_\infty = J_{[1,u_\infty,T_\infty]}(p) := \frac{1}{4\pi \mathfrak{c} T_\infty K_2(\frac{\mathfrak{c}^2}{T_\infty})} e^{-\frac{p_0u_{\infty,0} - p\cdot u_\infty}{T_\infty}},
\end{equation}
where $T_{\infty}>0$, $u_{\infty}=(u_{\infty,1},u_{\infty,2},u_{\infty,3})\in \mathbb {R}^3$, $u_{\infty,0} = \sqrt{\mathfrak{c}^2+|u_\infty|^2}$ are constants. We assume that $u_{\infty,2}=u_{\infty,3}=0$. 

By Lemma \ref{lem:Mach:number} and Remark \ref{rmk:Mach:number}, we define the sound speed $c_\infty$ and Mach number $\mathcal{M}^\infty$ of the far field equilibrium state by
\begin{equation}\label{eq:Mach:number}
    {c}_\infty := \sqrt{T_\infty \cdot \frac{z a_2(z) - a_1(z) }{z a_3(z)} },\qquad \mathcal{M}^\infty:=\frac{u_{\infty,1}}{c_\infty},
\end{equation}
where $z := \frac{\mathfrak{c}^2}{T_\infty}$, $a_i(z)$, $i=1,2,3$, are defined below in \eqref{eq:a1}-\eqref{eq:a3}. Note that the flow at infinity is incoming (resp. outgoing) if $\mathcal{M}^\infty<0$ (resp. $> 0$) and supersonic (resp. subsonic) if $|\mathcal{M}^\infty| > 1$ (resp. $< 1$).

\begin{remark}
    The relativistic sound speed $\hat{c}_\infty := \frac{c_\infty}{c_{\infty,0}}\mathfrak{c}=\frac{c_\infty}{\sqrt{\mathfrak{c}^2+c_\infty^2}}\mathfrak{c}$ is exactly the sound speed $c_s$ as expressed in equation (69) of Chapter I in \cite{Guichelaar}. From the definitions of $a_i$, $i=1,2,3$, in \eqref{eq:a1}-\eqref{eq:a3}, it is easy to check that 
    \[
        \lim_{z\to +\infty} \frac{z a_2(z) - a_1(z) }{z a_3(z)} = \frac{5}{3}, 
    \]
    which is consistent with the fact that the classical Boltzmann equation, in which case the sound speed is $\sqrt{\frac{5}{3}T}$, can be viewed as the limit of the relativistic Boltzmann equation as $\mathfrak{c}$ goes to infinity. Moreover, the relativistic sound speed conforms in the ultrarelativistic limit that
    \[
        \lim_{z\to 0} \frac{\hat{c}_\infty}{\mathfrak{c}} = \frac{1}{\sqrt{3}}. 
    \] 
\end{remark}

The number of solvability conditions for problem \eqref{eq:BVP} depends on $\mathcal{M}^\infty$. Define $n^+$ to be the number of solvability conditions:
\[
    n^+ := \left\{ \begin{array}{l}
        0,\quad \mathcal{M}^\infty\in (-\infty,-1),\\
        1,\quad \mathcal{M}^\infty\in (-1,0),\\
        4,\quad \mathcal{M}^\infty\in (0,1),\\
        5,\quad \mathcal{M}^\infty\in (1,\infty). \\
    \end{array} \right.
\]

We introduce for $\beta \in \mathbb{R}$ the weight function
\begin{equation}\label{eq:weight}
    W_\beta(p) := p_0^{-\beta} J_{[1,u_\infty,T_\infty]}^{\frac{1}{2}}(p), 
\end{equation}
and let $\eta$ be defined as
\begin{equation}\label{eq:eta}
    \eta := 1-\frac{1}{2}(3 | \varsigma | + a + 2b ) \in (0,1],
\end{equation}
where $\varsigma, a$ and $b$ are the parameters given above. 

The main theorem of this paper is as follows. 

\begin{theorem}\label{thm:main:result}
    Suppose that the scattering kernel $\sigma$ satisfies \eqref{eq:sigma}, \eqref{eq:ab} and \eqref{eq:gamma}, $\eta$ defined in \eqref{eq:eta} belongs to $[\frac{2}{3}, 1]$, $\beta > 2$ and $\mathcal{M}^\infty\ne 0,\pm1$. Then there exist positive numbers $\epsilon_0$, $\tau$, $c_0$ and a $C^1$ map
    \[
        \Psi:L^2(\mathbb{R}^3_+)\cap L_{p,\beta}^\infty(\mathbb{R}^3_+)\rightarrow\mathbb{R}_+^{n^+},\quad \Psi(0)=0, 
    \]
    such that the following holds. 

    {\rm (i)} For any $F_0$ satisfying
            
    \begin{equation}\label{eq:main:con1}
        |F_0(p)-J_\infty(p)|\le \epsilon_0 W_{\beta}(p),\quad p\in\mathbb{R}^3, p_1>0,
    \end{equation}
    and
    \begin{equation}\label{eq:main:con2}
        \Psi(W_0^{-1}(F_0-J_\infty))=0,
    \end{equation}
    problem \eqref{eq:BVP} has a unique solution $F$ in the class 
    \[
        |F(x,p)-J_\infty(p)|\le c_0 e^{-\tau x}W_{\beta}(p),\quad x\in \mathbb{R}^+,p\in \mathbb{R}^3.
    \]

    {\rm (ii)} The set of $F_0$ satisfying \eqref{eq:main:con1} and \eqref{eq:main:con2} forms a (local) $C^1$ manifold of codimension $n^+$.
\end{theorem}

\begin{remark}
    The condition $\beta>2$ in the above theorem can be weakened to $\beta > \frac{3+2(1-\eta) + \frac{a}{2}}{2}$ from the proof of Theorem \ref{thm:non:exist}, where $\eta\in[\frac{2}{3},1]$, \eqref{eq:sigma}, \eqref{eq:ab} and \eqref{eq:gamma} implies that $0\le a\le \frac{2}{3}$ and $\frac{3+2(1-\eta) + \frac{a}{2}}{2} \le 2$. In the particular case that $a = b = \varsigma = 0$, $\eta = 1$, the theorem holds for any $\beta > 3/2$. 
\end{remark}

In this paper, we explicitly address for the relativistic Maxwellian in the far field the sound speed in \eqref{eq:Mach:number}, which is derived from the eigenvalues of the operator $P_0 \hat{p}_1 P_0$ corresponding to the relativistic Maxwellian with nonzero macroscopic velocities. 
The primary challenge of this paper arises from two key factors. First, the inherent complexity of the collision kernel in the relativistic Boltzmann equation, as well as the moments of the relativistic Maxwellian, presents substantial difficulties compared to the non-relativistic case. Second, obtaining $L^\infty$ estimates for the solutions requires the careful selection of a weight function that is singular in velocity (at $p_1=0$ and $|p| \gg 1$) and exhibits exponential spatial decay. The choice of this weight function further imposes a constraint on the parameter $\eta\in[\frac{2}{3},1]$. 
To better monitor the dimensionless quantities in the calculation, we consistently include the speed of light $\mathfrak{c}$ in all computations.

The organization of this paper is as follows.
In section \ref{sec:notations:prop}, we denote some notations and prove that the operator $L$ and $\Gamma$ corresponding to the general (i.e. with nonzero bulk velocity) relativistic Maxwellian $J_\infty$ have the same properties as the ones corresponding to the standard (i.e. with vanishing bulk velocity) relativistic Maxwellian $J_\infty^0$ by using Lorentz transformations. And we also present some properties of $P_0\hat{p}_1P_0$. In section \ref{sec:elp}, we prove the existence and uniqueness of solution for the linearized relativistic Boltzmann equation with damping term by energy estimate and Riesz representation theorem. In section \ref{sec:we}, we give the weighted estimates of the solution for the linear problem with damping term. In section \ref{sec:enp}, we prove the existence of the solution for the nonlinear equation with damping term. In section \ref{sec:without:damping}, we discuss the existence of solutions to the problem without damping under suitable solvability conditions.

\section{Some notations and properties}\label{sec:notations:prop}

\subsection{Notations in the relativistic system}\label{sub:notations}

In this paper, we use $\| \cdot \|_{p}$ and $\| \cdot \|_{x,p}$ to denote the standard norm in $L^2(\mathbb{R}^3)$ and $L^2\big( (0,+\infty)\times \mathbb{R}^3 \big)$, respectively. Use $\| \cdot \|$ for convenience when the independent variables are obvious. Let $\| \cdot \|_{\pm}$ to denote the norm in $L^2(\mathbb{R}^3_{\pm})$ with $\mathbb{R}^3_{+} := \{p\in \mathbb{R}^3 \mid p_1>0\}$, $\mathbb{R}^3_{-} := \{p\in \mathbb{R}^3 \mid p_1<0\}$. 
Use $( \cdot, \cdot )$ to denote the inner product in $L^2_{x,p}( (0, \infty)\times \mathbb{R}^3)$ and $\langle \cdot, \cdot \rangle_\pm$ to denote the inner product in $L_p^2(\mathbb R^3_\pm)$. 

Use $\| \cdot \|_{L^\infty_{p,\beta}}$ and $\| \cdot \|_{L^\infty_{x,p,\beta}}$ to denote for any $\beta\in\mathbb{R}$ the weighted norm
\[
    \| f \|_{L^\infty_{p,\beta}} := \| p_{0}^{\beta} f \|_{L_{p}^\infty} = \underset{p \in \mathbb{R}^3}{\sup} | p_{0}^{\beta} f( p )|, \qquad 
    \| g \|_{L^\infty_{x,p,\beta}} := \| p_{0}^{\beta} g \|_{L_{x,p}^\infty} = \underset{x \in \mathbb{R}^+, p \in \mathbb{R}^3}{\sup} | p_{0}^{\beta} g( x,p )|.
\]
Let $\| \cdot \|_{L^\infty_{p,+}}$ and $\| \cdot \|_{L^\infty_{p,\beta}(\mathbb{R}^3_+)}$ represent the (weighted) norms in half space: 
\[
    \| f \|_{L^\infty_{p,+}} := \| f \|_{L_{p}^\infty(\mathbb{R}^3_+)} = \underset{p \in \mathbb{R}^3, p_1>0}{\sup} | f( p )|,
\]
\[
    \| f \|_{L^\infty_{p,\beta}(\mathbb{R}^3_+)} := \| p_{0}^{\beta} f \|_{L_{p}^\infty(\mathbb{R}^3_+)} = \underset{p \in \mathbb{R}^3, p_1>0}{\sup} | p_{0}^{\beta} f( p )|.
\]

Throughout this paper, we use $c$ to represent a constant depending on $u_\infty,T_\infty, \mathfrak{c}$ and the parameters in \eqref{eq:sigma}, and the specific value it represents may vary from one line to another.

For $p,q\in\mathbb{R}^3$, we let $p\cdot q := \sum_{i=1}^{3} p_i q_i$. Denote 
\begin{equation}\label{eq:D}
    D := 
    \begin{pmatrix}
        1 & 0 & 0 & 0 \\
	0 & -1 & 0 & 0 \\
	0 & 0 & -1 & 0 \\
	0 & 0 & 0 & -1 \\
    \end{pmatrix}.
\end{equation}
Define the Lorentz inner product of any two four-vectors $P$ and $Q$ as
\[
    P \cdot Q := P D Q = p_0 q_0 - \sum_{i=1}^{3} p_i q_i.
\]

Let $s$ be the square of the energy given by
\begin{equation}\label{eq:s}
    s(P,Q) := |P+Q|^2 = (P+Q) \cdot (P+Q) = 2 ( P \cdot Q + \mathfrak{c}^2 ). 
\end{equation}
The relative momentum $g \ge 0$ is defined as
\begin{equation}\label{eq:g}
    2 g(P,Q) := |P-Q| = \sqrt{-(P-Q)\cdot(P-Q)}=\sqrt{2(P\cdot Q-\mathfrak{c}^2)}.
\end{equation}

A direct calculation shows that 
\begin{equation}\label{eq:sg}
    s = 4 g^2 + 4 \mathfrak{c}^2.
\end{equation}
It is easy to see that 
\[
    \sqrt{s} \le 2(\mathfrak{c}+g) \le \sqrt{2s}.
\]
Equations \eqref{eq:conservation:laws}, \eqref{eq:s} and \eqref{eq:sg} together lead to the conclusion that both $s$ and $g$ are collision invariants, that is to say, $s(P,Q)=s(P',Q')$ and $g(P,Q)=g(P',Q')$.
	
In the center-of-momentum system, the post-collisional velocities in (\ref{eq:Q}) can be expressed for some $\omega\in \mathbb{S}^2$ as
\begin{equation}\label{eq:post:velocity:boost}
    \left\{ 
    \begin{array}{l}
        \displaystyle p' = \frac{1}{2} ( p + q ) + g(P,Q) \cdot \big( \omega + ( \tilde{\gamma}-1 ) ( p+q ) \frac{ ( p + q ) \cdot \omega}{ | p + q |^2} \big), \\
        \displaystyle q' = \frac{1}{2} ( p + q ) - g(P,Q) \cdot \big( \omega + ( \tilde{\gamma}-1 ) ( p+q ) \frac{ ( p+q ) \cdot \omega}{ | p+q |^2} \big),
    \end{array} 
    \right.
\end{equation}
where $\tilde{\gamma} := \frac{(p_0+q_0)}{\sqrt{s}} $. The post-collisional energies are given by
\begin{equation}\label{eq:post:energy:boost}
    \left\{ 
	\begin{array}{l}
        \displaystyle p_0' = \frac{1}{2} ( p_0+q_0 ) +\frac{g}{\sqrt{s}}(p+q)\cdot \omega,\\
        \displaystyle q_0' = \frac{1}{2}( p_0+q_0 ) -\frac{g}{\sqrt{s}}(p+q)\cdot \omega.
	\end{array} 
    \right.    
\end{equation}

The scattering angle $\theta$ is defined by
\begin{equation}\label{eq:scattering:angle}
	\cos \theta := -\frac{(P-Q)\cdot(P'-Q')}{4g^2}.
\end{equation}
Note that the conservation of momentum and energy in \eqref{eq:conservation:laws} ensures that the right-hand side of \eqref{eq:scattering:angle} is less than one, hence \eqref{eq:scattering:angle} gives a good definition of $\theta$; See \cite{Glassey}. 
	
Using the Lorentz transformations described in \cite{deGroot:vanLeeuwen:vanWeert}, four of the integrations in \eqref{eq:Q} can be carried out in the center-of-momentum system to get rid of the delta functions and obtain that
\begin{equation}\label{eq:collision:operator}
    \begin{aligned}
        Q ( F_1, F_2 ) & = \int_{\mathbb{R}^3} \int_{\mathbb{S}^2} v_{\text{\o}}(p, q) \sigma ( g,\theta ) \big( F_1 ( p' ) F_2 ( q' ) - F_1( p ) F_2 ( q ) \big) d\omega dq \\
        & =: Q_+ (F_1, F_2) - Q_- (F_1, F_2).
    \end{aligned}
\end{equation}
Here $v_{\text{\o}}$ is the M{\o}ller velocity
\[
    v_{\text{\o}}(p, q) := \frac{\mathfrak{c}}{2}\sqrt{\Big| \frac{p}{p_0}-\frac{q}{q_0}\Big| ^2 - \Big| \frac{p}{p_0} \times \frac{q}{q_0}\Big| ^2}=\frac{\mathfrak{c}}{2}\frac{g\sqrt{s}}{p_0q_0}.
\]
For the readers' convenience, we provide some useful lemmas on the Lorentz transformations in Appendix \ref{Lorentz}. Note that the angle $\theta$ in \eqref{eq:scattering:angle} depends only on $P,Q$ and their collision angle. By Lemma \ref{lem:pbar} and Lemma \ref{lem:post:tilde}, $\theta$ can also be expressed as
\begin{equation}\label{eq:theta}
    \cos{\theta} = \frac{\bar{p}}{|\bar{p}|} \cdot \omega,
\end{equation}
where $\bar{p}$ and $\omega$ are defined by 
\begin{equation}\label{eq:p:bar}
    \Lambda(P-Q)=(0,\bar{p})^T, \qquad \Lambda(P'-Q')=(0,2g\omega)^T,
\end{equation} 
and $\Lambda$ is a Lorentz transformation which satisfies $\Lambda(P+Q) = (\sqrt{s},0,0,0)^T$. If we choose $\Lambda$ to be the boost matrix of $P+Q$, then $\omega$ in \eqref{eq:p:bar} is exactly the one in the center of momentum system \eqref{eq:post:velocity:boost} and \eqref{eq:post:energy:boost}. Then one can write $d\omega = \sin {\theta}d\theta d\psi$ with $0\le\theta\le\pi$ and $0\le\psi\le2\pi$. 

It is shown in \cite{deGroot:vanLeeuwen:vanWeert} that
\[
    Q(F,F) \equiv 0 \quad \Longleftrightarrow \quad F = J_{[ \rho, u,T ]}\mbox{ for some $\rho, u$ and $T$}.
\]
If the mass density $\rho$, fluid bulk velocity $u=(u_1,u_2,u_3) \in \mathbb{R}^3$ and temperature $T>0$ are functions of $t$ and $x$, the distribution function $J$ is referred to as {\it local relativistic Maxwellians}. In the event that the variables $\rho$, $u$ and $T$ are constants, then $J$ is referred to as {\it global relativistic Maxwellians}.
	
A function $\phi(p)$ is called a {\it collision invariant} of $Q$ if it holds for any $F\in L^2(\mathbb{R}^3)$ that
\[
    \langle \phi, Q(F,F) \rangle = 0,
\]
where $\langle \cdot, \cdot \rangle$ is the inner product in $L^2_p(\mathbb R^3)$. It is evident that $Q$ has five collision invariants 
\begin{equation}\label{eq:invariants}
    \phi_0=1,\qquad \phi_i=p_i, \quad i=1,2,3, \qquad \phi_4=p_0,
\end{equation}
which indicate the conservation of mass, momentum and energy, respectively, during the binary collision of particles. 

\subsection{Properties of $L$ and $\Gamma$}\label{sub:Properties:L:Gamma}

Let
\begin{equation}\label{eq:J0inf}
    J_\infty^0 := J_{[1, 0, T_\infty]}(p) = \frac{1}{4\pi \mathfrak{c} T_\infty K_2(\frac{\mathfrak{c}^2}{T_\infty})} e^{-\frac{\mathfrak{c} p_0}{T_\infty}}
\end{equation}
be the relativistic Maxwellian with zero bulk velocity. 

The existence of solution to \eqref{eq:BVP} will be studied by perturbation theory for the linearized problem. Let the linear operator $L$ and nonlinear operator $\Gamma$ be defined through
\begin{align*}
    L(f) & := J_\infty^{-\frac{1}{2}} [ Q(J_\infty,J_\infty^{\frac{1}{2}}f)+Q(J_\infty^{\frac{1}{2}}f,J_\infty) ], \\
    \Gamma(f,f) & := J_\infty^{-\frac{1}{2}}Q(J_\infty^{\frac{1}{2}}f,J_\infty^{\frac{1}{2}}f),
\end{align*}
where $J_\infty$ is the global Maxwellian with nonzere bulk velocity defined in (\ref{eq:J:infty}).
	
Let $\mathcal{N} := \text{ker}(L)$, the null space of $L$, then 
\begin{equation}\label{eq:Null:L}
    \mathcal{N} = \text{span} \big\{ J_\infty^\frac{1}{2}\phi_i(p) \big\}, \quad i=0,1,\cdots,4,
\end{equation}
where $\phi_i(p)$ are the collision invariants in (\ref{eq:invariants}). So $\mathcal{N}$ can be viewed as a 5-dimensional subspace of $L_p^2(\mathbb R^3)$. Let $\mathcal{N}^\perp$ be the orthogonal complement of $\mathcal{N}$, and let
\begin{equation}\label{eq:mapping}
    P_0: L_p^2(\mathbb R^3)\rightarrow \mathcal{N}, \qquad P_1: L_p^2(\mathbb R^3)\rightarrow \mathcal{N}^\perp
\end{equation}
be the orthogonal projection operators.

Now we present some properties of the linear operator $L$ and bilinear operator $\Gamma$.
\begin{proposition}\label{prop:L}
    Suppose that the scattering kernel $\sigma$ satisfies \eqref{eq:sigma}, \eqref{eq:ab} and \eqref{eq:gamma}. Then 
    
    {\rm (i)} $L$ can be decomposed as $L = -\nu (p) + K$, where  
    \[
        \nu(p) := \int_{\mathbb{R}^3} \int_{\mathbb{S}^2} v_{\text{\rm \o}} \sigma(g, \theta) J_\infty(q) d\omega dq 
        > 0, 
    \]
    satisfying
    \begin{equation}\label{eq:nu}
        \nu_0^{-1} p_0^\frac{a}{2} \le \nu(p) \le \nu_0 p_0^\frac{a}{2},
    \end{equation}
    for some constant $\nu_0>0$ depending only on $\mathfrak{c}, u_{\infty}, T_{\infty}$ and $a$ in \eqref{eq:sigma}. 
        
    {\rm (ii)} $K$ is a compact operator in the form 
    \[
        K f(p) = \int_{\mathbb{R}^3} k(p,q) f(q) dq, 
    \]
    which has the regularizing property that it is bounded as an operator
    \[
        K: L_p^2 \rightarrow L_p^2 \cap L_p^\infty,  \qquad K: L_{p,\beta}^\infty \rightarrow L_{p,\beta+\eta}^\infty
    \]
    for all $\beta \ge 0$ and $\eta$ as defined in \eqref{eq:eta}. 
    
    {\rm (iii)} $L$ is non-positive and self-adjoint on $L_p^2$. For any $f\in L_p^2$, it holds that 
    \begin{equation}\label{eq:L:es}
        \langle  f, Lf  \rangle \le - c \langle \nu(p) P_1 f, P_1 f \rangle, 
    \end{equation}
    where $c>0$ is a constant independent of $f$ and $P_1$ is the projection operator onto $\mathcal{N}^\perp$.
\end{proposition}

\begin{proof}
    If $u_\infty = 0$, then $J_\infty$ becomes $J_\infty^0$ in \eqref{eq:J0inf}, and the proposition  for $J_\infty^0$ was proved in \cite{Glassey:Strauss}. We only need to present the proof for $u_\infty \not= 0$.
    
    Let $\nu^0(p)$ and $k^0(p,q)$ be the corresponding functions for $J_\infty^0$. 
	
    (i) Any Lorentz transformation $\Lambda$ ensures that for any two four-vectors $P$ and $Q$, 
    \[
        P \cdot Q = (\Lambda P) \cdot (\Lambda Q). 
    \]
    
    By Remark \ref{rmk:Lambda:U} in the appendix, let $\Lambda_U$ be a Lorentz transformation, depending only on $U$, such that $\Lambda_U U = (\sqrt{U\cdot U},0,0,0) = (\mathfrak{c},0,0,0)^T$. 
    Let $\tilde{Q} := \Lambda_U Q$, $\tilde{P} := \Lambda_U P$. By Lemma \ref{lem:Lorentz:inv} in the appendix, $\frac{dq}{q_0}$ is a Lorentz invariant measure, i.e. $\frac{dq}{q_0}=\frac{d\tilde{q}}{\tilde{q}_0}$. Hence 
    \begin{align*}
        \nu(p) = & \int_{\mathbb{R}^3} \int_{\mathbb{S}^2} v_{\text{\o}} \sigma(g,\theta) J_\infty(q) d \omega d q \\ 
        = & \frac{\mathfrak{c}}{2}\cdot\frac{1}{4\pi \mathfrak{c} T_\infty K_2(\frac{\mathfrak{c}^2}{T_\infty})} \cdot\frac{1}{p_0} \int_{\mathbb{R}^3} \int_{\mathbb{S}^2} g(P,Q)\sqrt{s(P,Q)} \sigma(g(P,Q),\theta(P,Q,\omega) )  e^{-\frac{(\Lambda_U Q)\cdot (\Lambda_U U)}{T_\infty}} d \omega \frac{d\Tilde{q}}{\Tilde{q}_0}. 
    \end{align*}
    Now we consider the changes of angle $\theta$ under Lorentz transformation. By \eqref{eq:theta}, \eqref{eq:p:bar}, the definition of $\tilde{P}$, $\tilde{Q}$, Lemma \ref{lem:LT:mul} and Lemma \ref{lem:theta}, we have that
    \begin{align*}
        \cos{\theta}&=\frac{\bar{p}}{|\bar{p}|} \cdot \omega \\
        &=-\frac{[\Lambda_{P+Q}(P-Q)]\cdot(0,\omega)}{ 2 g } \\
        &=-\frac{[\Lambda_{P+Q}\Lambda_{U}^{-1}\Lambda_{U}(P-Q)]\cdot(0,\omega) }{ 2 g } \\
        &= -\frac{[\Lambda_{\Lambda_{U}(P+Q)}\Lambda_{U}(P-Q)]\cdot[\Lambda_{\Lambda_{U}(P+Q)}\Lambda_{U}\Lambda_{P+Q}^{-1}(0,\omega)]}{ 2 g } \\  
        &=-\frac{[\Lambda_{\tilde{P}+\tilde{Q}}(\tilde{P}-\tilde{Q})] \cdot ( 0, \tilde{\omega} ) }{ 2 g },  
    \end{align*}
    By Lemma \ref{lem:pbar}, we denote $\Lambda_{\tilde{P}+\tilde{Q}} (\tilde{P}-\tilde{Q}) =: (0, \bar{\tilde{p}})$ with $|\bar{\tilde{p}}| = 2 g$. Therefore, $\cos\theta = \bar{\tilde{p}} \cdot \tilde{\omega} / |\bar{\tilde{p}}|$, $\theta(P,Q,\omega) = \theta(\tilde{P}, \tilde{Q},\tilde{\omega})$. Note that $g$ and $s$ are invariant under any Lorentz transformations. Therefore, 
    \[
        \nu(p) = 
        \frac{\mathfrak{c}}{2}\cdot\frac{1}{4\pi \mathfrak{c} T_\infty K_2(\frac{\mathfrak{c}^2}{T_\infty})} \cdot\frac{1}{p_0} \int_{\mathbb{R}^3} \int_{\mathbb{S}^2} g(\Tilde{P},\Tilde{Q})\sqrt{s(\Tilde{P}, \Tilde{Q})}
        \sigma(g(\Tilde{P},\Tilde{Q}),\theta( \tilde{P}, \tilde{Q}, \tilde{\omega} ) )  e^{-\frac{\mathfrak{c}\Tilde{q}_0}{T_\infty}} d \tilde{\omega} \frac{d\Tilde{q}}{\Tilde{q}_0}. 
    \]
        
    Swap the notations $\tilde{q}$ and $q$, $\tilde{\omega}$ and $\omega$, we write $\nu(p)$ as
    \begin{equation}\label{eq:nu:in:lem}
        \begin{aligned}
            \nu(p)= & \frac{\mathfrak{c}}{2} \cdot \frac{1}{4\pi \mathfrak{c} T_\infty K_2(\frac{\mathfrak{c}^2}{T_\infty})} \cdot\frac{1}{p_0} \int_{\mathbb{R}^3} \int_{\mathbb{S}^2} g(\tilde{P},Q) \sqrt{ s(\tilde{P}, Q)} \sigma(g(\tilde{P}, Q), \theta( \tilde{P}, Q, \omega )) e^{-\frac{\mathfrak{c}q_0}{T_\infty}} d \omega \frac{dq}{q_0} \\
            = & \frac{\tilde{p}_0}{ p_0} \nu^0(\tilde{p}), 
        \end{aligned}
    \end{equation}
    where $\nu^0(p)$ corresponds to the standard relativistic Maxwellian $J_\infty^0$ in \eqref{eq:J0inf}. 
    From \cite{Glassey:Strauss}, we know that there exists a constant $\nu_0>0$ such that
    \begin{equation}\label{eq:nu0}
        \nu_0^{-1} p_0^\frac{a}{2} \le \nu^0 (p) \le \nu_0 p_0^\frac{a}{2}, 
    \end{equation}
    with $a$ being the coefficient in \eqref{eq:sigma}. 
        
    From Lemma \ref{lem:tildepq:equiv}, there exist positive constants $c_1$ and $c_2$, depending only on $\Lambda_U$, such that 
    \begin{equation}\label{eq:p0:equivalence}
        0 < c_1 \le \frac{\Tilde{p}_0}{p_0}, \frac{\Tilde{q}_0}{q_0} \le c_2. 
    \end{equation}
    Combining \eqref{eq:nu:in:lem}-\eqref{eq:p0:equivalence}, we can obtain estimate \eqref{eq:nu}. 
        
    (ii)
    By definition, 
    \[
        \begin{aligned}
            Kf(p) & = \int_{\mathbb{R}^3} \int_{\mathbb{S}^2} v_{\text{\o}} \sigma(g,\theta) W_0(q)[W_0(q')f(p')+W_0(p')f(q')-W_0(p)f(q)] d\omega dq\\
            & =: K_2f(p)-K_1f(p)\\
            &= \int_{\mathbb{R}^3} k_2(p,q) f(q)dq-\int_{\mathbb{R}^3} k_1(p,q) f(q)dq. 
        \end{aligned}
    \]
    Similar as in (i), let $\Lambda_U$ be a Lorentz transformation such that $\Lambda_U U = (\mathfrak{c},0,0,0)^T$, $\tilde{Q} := \Lambda_U Q$, $\tilde{P} := \Lambda_U P$. Then we have that
    \[
        \begin{aligned}
            K_1 f(p) & = \int_{\mathbb{R}^3} \int_{\mathbb{S}^2} v_{\text{\o}} \sigma(g,\theta) W_0(q)W_0(p)f(q) d\omega dq \\ 
        & = \frac{\mathfrak{c}}{2} \cdot \frac{1} 
            {4\pi \mathfrak{c} T_\infty K_2(\frac{\mathfrak{c}^2}{T_\infty})} \cdot \frac{1}{p_0} \int_{\mathbb{R}^3} f(q) \frac{dq}{q_0} \int_{\mathbb{S}^2} g(P,Q) \sqrt{s(P,Q)} \sigma(g(P,Q), \theta(P,Q,\omega) ) e^{-\frac{(P+Q) \cdot U}{2T_\infty}} d \omega  \\
        & = \frac{\mathfrak{c}}{2} \cdot \frac{1} 
            {4\pi \mathfrak{c} T_\infty K_2(\frac{\mathfrak{c}^2}{T_\infty})} \cdot \frac{1}{p_0} \int_{\mathbb{R}^3} f(q) \frac{dq}{q_0} \int_{\mathbb{S}^2} g(\tilde{P},\tilde{Q}) \sqrt{s(\tilde{P},\tilde{Q})} \sigma(g(\tilde{P},\tilde{Q}),\theta(\tilde{P},\tilde{Q},\tilde{\omega}) ) e^{-\frac{\mathfrak{c} (\tilde{p}_0+\tilde{q}_0)}{2T_\infty}} d \tilde{\omega}  \\
        & = \int_{\mathbb{R}^3} \frac{\tilde{p}_0 \tilde{q}_0}{p_0 
        q_0} k_1^0(\tilde{p},\tilde{q}) f(q)dq.
        \end{aligned}
    \]
    Here
    \[
            k_1^0(p,q) = \frac{\mathfrak{c}}{2} \cdot \frac{1}{4\pi \mathfrak{c} T_\infty K_2(\frac{\mathfrak{c}^2}{T_\infty})} \cdot \frac{1}{p_0 q_0} \int_{\mathbb{S}^2} g(P,Q) \sqrt{s(P,Q)} \sigma(g(P,Q),\theta) e^{-\frac{\mathfrak{c} (p_0+q_0)}{2T_\infty}} d \omega, 
    \]
    which corresponds to the standard relativistic Maxwellian $J_\infty^0$. Hence, 
    \[
        k_1(p,q) = \frac{\tilde{p}_0 \tilde{q}_0}{ p_0 q_0 }k_1^0(\tilde{p},\tilde{q}). 
    \]

    Similarly, by using Lorentz transformation and the expression of $k_2(p,q)$ in \cite{Marek:Maria}, we have that
    \begin{equation}\label{eq:k2}
        k_2(p,q)=\frac{\tilde{p}_0 \tilde{q}_0}{ p_0 q_0 }k_2^0(\tilde{p},\tilde{q}), 
    \end{equation}
    where $k_2^0(p,q)$ is the corresponding kernel function for the standard relativistic Maxwellian $J_\infty^0$.
    By \eqref{eq:p0:equivalence}, we know that $k_1(p,q)$ and $k_2(p,q)$ have the same properties as $k_1^0(p,q)$ and $k_2^0(p,q)$ as in \cite{Glassey:Strauss}. 

    (iii)
    From the decomposition of $L$ and the properties of $K$, it is known for any $f \in \mathcal{N}^\perp$ that 
    \[
        \langle  f,Lf  \rangle \le - c_0 \langle f,f \rangle,
    \]
    where $c_0$ is a positive constant independent of $f$. By (ii), we have that
    \[
        L = - \nu(p) + K,
    \]
    and $K$ is a compact operator from $L^2_p$ to itself such that there exists a constant $c>0$ such that $| \langle  f, K f \rangle | \le c \langle  f, f  \rangle$. For any $c>0$, we can choose a small constant $\epsilon > 0$, such that $(1-\epsilon) c_0 - c \epsilon > 0$. Then for any $f\in L^2_p$, it holds that 
    \[
    \begin{aligned}
        - \langle  f,Lf  \rangle & = - \langle  P_1f,L(P_1f)  \rangle\\
        &=-\langle  P_1f,(1-\epsilon)L(P_1f) \rangle-\langle  P_1f,\epsilon L(P_1f)  \rangle \\
        &\ge (1-\epsilon) c_0 \langle  P_1f, P_1f  \rangle+\epsilon \langle  P_1f, \nu(p)P_1f  \rangle-c\epsilon \langle  P_1f, P_1f  \rangle\\
        &\ge \epsilon \langle  P_1f, \nu(p)P_1f  \rangle. 
    \end{aligned}
    \]
    Hence (iii) holds. 
\end{proof}

\begin{lemma}[\cite{Glassey:Strauss}]\label{lem:sigma:estimate}
    There exists a constant $c > 0$, depending on $T_\infty$, $\mathfrak{c}$ and the parameters in \eqref{eq:sigma}, \eqref{eq:ab} and \eqref{eq:gamma}, such that 
    \[
         \int_{\mathbb{R}^3} \int_{\mathbb{S}^2} \sigma(g, \theta) J_\infty(q)^0 d\omega dq = c \int_{\mathbb{R}^3} \int_{\mathbb{S}^2} \sigma( g, \theta ) e^{-\frac{\mathfrak{c} q_0}{2 T_\infty}}  d\omega dq	\le c p_0^{\frac{a}{2}}.
    \]
\end{lemma}

\begin{lemma}[\cite{Glassey:Strauss}]\label{lem:p0'q0'p0}
    Let $p_0'$, $q_0'$ be the post-collisional energies of $p$ and $q$. For any $\beta \ge 0$, it holds for some $c_\beta > 0$ that
    \[
        (p_0')^\beta (q_0')^\beta \ge c_\beta (p_0)^\beta.
    \]
\end{lemma}

\begin{proposition}\label{prop:gamma}
    Suppose that the scattering kernel $\sigma$ satisfies \eqref{eq:sigma}, \eqref{eq:ab} and \eqref{eq:gamma}. 
    Then the projection of $\Gamma(h_1, h_2)$ on $\mathcal{N}$, the null space of $L$, vanishes and there exists a constant $c>0$ such that for any $\beta \ge 0$, 
    \[
        \begin{aligned}
            \| \nu^{-1} \Gamma(h_1,h_2) \|_{L^\infty_{x,p,\beta}}\le c \| h_1 \|_{L^\infty_{x,p,\beta}} \| h_2 \|_{L^\infty_{x,p,\beta}}. 
        \end{aligned}
    \]
\end{proposition}
\begin{proof}
    By definition, 
    \[
    \begin{aligned}
        \Gamma(h_1,h_2)&=W_0^{-1}(p)Q(W_0 h_1,W_0 h_2)\\
        & = \int_{\mathbb{R}^3} \int_{\mathbb{S}^2} v_{\text{\o}} \sigma(g,\theta) W_0(q)[h_1(p')h_2(q')-h_1(p)h_2(q)] d \omega d q. 
    \end{aligned}
    \]
    In order to get the conclusion for general Maxwellian, according to Lemma \ref{lem:sigma:estimate} and Lemma \ref{lem:tildepq:equiv}, letting $\tilde{Q} := \Lambda_U Q$, $\tilde{P} := \Lambda_U P$, we have that
    \[
    \begin{aligned}
         \int_{\mathbb{R}^3} \int_{\mathbb{S}^2}\sigma(g,\theta) W_0(q) d\omega dq & = \Big( \frac{1}{4\pi \mathfrak{c} T_\infty K_2(\frac{\mathfrak{c}^2}{T_\infty})} \Big)^{1/2}  \int_{\mathbb{R}^3}  \int_{\mathbb{S}^2} \sigma(g,\theta(P,Q,\omega)) e^{-\frac{Q\cdot U}{2 T_\infty}} q_0 d\omega \frac{dq}{q_0}\\
        & = \Big( \frac{1}{4\pi \mathfrak{c} T_\infty K_2(\frac{\mathfrak{c}^2}{T_\infty})} \Big)^{1/2} \int_{\mathbb{R}^3} \int_{\mathbb{S}^2} \sigma(g,\theta(\tilde{P},\tilde{Q},\tilde{\omega})) e^{-\frac{\mathfrak{c} \tilde{q}_0}{2 T_\infty}} q_0 d\tilde{\omega} \frac{d\tilde{q}}{\tilde{q}_0}\\
        &\le \frac{c}{c_1} \int_{\mathbb{R}^3} \int_{\mathbb{S}^2} \sigma(g,\theta(\tilde{P},\tilde{Q},\tilde{\omega})) e^{-\frac{\mathfrak{c} \tilde{q}_0}{2 T_\infty}} d\tilde{\omega} d\tilde{q}\\
        &\le c \tilde{p}_0^{\frac{a}{2}} \le c p_0^{\frac{a}{2}} \le c \nu(p). 
    \end{aligned}
    \]
    Now we claim that $v_{\text{\o}}$ is bounded. Actually, it is easy to prove that
    \[
        p_0q_0 \ge \mathfrak{c}^2 + |p\cdot q|. 
    \]
    Therefore, 
    \[
        2 g \sqrt{g^2+\mathfrak{c}^2} \le 2g^2+\mathfrak{c}^2 \le 2 (g^2+\mathfrak{c}^2) = p_0 q_0 - p \cdot q + \mathfrak{c}^2 \le 2 p_0 q_0, 
    \]
    \[
        |v_{\text{\o}}|=\frac{\mathfrak{c}g\sqrt{s}}{2p_0q_0}=\frac{\mathfrak{c}g\sqrt{g^2+\mathfrak{c}^2}}{p_0q_0}\le \mathfrak{c}.
    \]
    By Lemma \ref{lem:p0'q0'p0}, we have that
    \[
    \begin{aligned}
        p_0^{\beta}|\Gamma(h_1,h_2)| &\le \frac{1}{c_\beta}\int_{\mathbb{S}^2} \int_{\mathbb{R}^3} v_{\text{\o}} \sigma(g,\theta) W_0(q) (p_0')^{\beta} (q_0')^{\beta} |h_1(p')h_2(q')|dqd\omega\\
        &+ \int_{\mathbb{S}^2} \int_{\mathbb{R}^3} v_{\text{\o}} \sigma(g,\theta) W_0(q) p_0^{\beta} |h_1(p)h_2(q)|dqd\omega\\
        &\le c \|h_1\|_{L^\infty_{p,\beta}(L^\infty_x)} \|h_2\|_{L^\infty_{p,\beta}(L^\infty_x)} \int_{\mathbb{S}^2} \int_{\mathbb{R}^3} \sigma(g,\theta) W_0(q)dqd\omega\\
        &\le c \nu(p) \|h_1\|_{L^\infty_{p,\beta}(L^\infty_x)} \|h_2\|_{L^\infty_{p,\beta}(L^\infty_x)}. 
    \end{aligned}
    \]
    This completes the proof of the proposition. 
\end{proof}

\begin{proposition}\label{prop:k:weight}
    Assume the scattering kernel $\sigma$ satisfying \eqref{eq:sigma}, \eqref{eq:ab} and \eqref{eq:gamma}. Then the integral kernel $k(p,q)$ is symmetric and satisfies
    
    {\rm (i)}
    \begin{equation}\label{eq:k:es}
        k(p,q)\le\frac{c(\mathfrak{c}+|q|)^{\frac{3|\varsigma|+a+2b}{2}}e^{-\frac{c\mathfrak{c}|p-q|}{T_\infty}}}{\big[|p\times q|^2 + \mathfrak{c}^2 |p-q|^2\big]^{\frac{1}{2}}|p-q|^{b+|\varsigma|}},
    \end{equation}
    where $c$ depends on $u_\infty, T_\infty, \mathfrak{c}$ and the parameters in \eqref{eq:sigma}.

    {\rm (ii)} $\underset{q}{\sup}\int|k(p,q)|dp<\infty$,
        
    {\rm (iii)} $\underset{q}{\sup}\int k^2(p,q)dp<\infty$,
        
    {\rm (iv)} $\int |k(p,q)|( \mathfrak{c}^2 + |p|^2)^{\frac{\alpha}{2}}dp \le c (\mathfrak{c}^2 + |q|^2)^{\frac{(\alpha - \eta)}{2}}$, for any $\alpha \in \mathbb{R}$, where $\eta\in (0,1]$ is defined in \eqref{eq:eta}. 
\end{proposition}
\begin{proof}
    According to the proof of (ii) in Proposition \ref{prop:L}, we obtain that $ k(p,q) \sim k^0(\tilde{p},\tilde{q})$. Therefore $k(p,q)$ has the same properties as $k^0(\tilde{p},\tilde{q})$, which is the kernel corresponding to the standard relativistic Maxwellian $J_\infty^0$. 
    
    By \eqref{eq:k2} and Lemma \ref{lem:tildepq:equiv}, we know that $g$ and $s$ are invariants, $|p-q| \sim |\tilde{p}-\tilde{q}|$, 
    $\mathfrak{c}+|q| \sim q_0 \sim \tilde{q}_0 \sim \mathfrak{c}+|\tilde{q}|$. 
    Following the same process of proving Lemma 3.7 in \cite{Glassey:Strauss}, we know that (i) holds. According to Lemma 3.8 in \cite{Glassey:Strauss}, we know that (ii) (iii) and (iv) for $\alpha\le 0$ holds. It remains to prove (iv) in the case $\alpha > 0$. 
        
    When $\alpha > 0$ in (iv), we split the integral into two parts
    \[
    \begin{aligned}
        \int |k(p,q)|(\mathfrak{c}^2+|p|^2)^{\frac{\alpha}{2}}dp & = \int_{|p|>2|q|} |k(p,q)|(\mathfrak{c}^2+|p|^2)^{\frac{\alpha}{2}}dp + \int_{|p|<2|q|} |k(p,q)|(\mathfrak{c}^2+|p|^2)^{\frac{\alpha}{2}}dp \\
        & =: I_1+I_2.
    \end{aligned}
    \]
    When $|p|>2|q|$, we have $|p-q|\ge|p|-|q|>|q|$. Then by \eqref{eq:k:es}, it holds that
    \[
    \begin{aligned}
        I_1 & \le c(\mathfrak{c}+|q|)^{\frac{3|\varsigma| +a+2b}{2}}\int_{|p|>2|q|} \frac{e^{-\frac{c\mathfrak{c}|p-q|}{T_\infty}}}{\big[|p\times q|^2+\mathfrak{c}^2|p-q|^2\big]^{\frac{1}{2}}|p-q|^{b+|\varsigma|}}(\mathfrak{c}^2+|p|^2)^{\frac{\alpha}{2}}dp \\
        & \le c(\mathfrak{c}+|q|)^{\frac{3|\varsigma|+a+2b}{2}}\int_{|p|>2|q|} \frac{e^{-\frac{c\mathfrak{c}|p-q|}{T_\infty}}}{\big[|p\times q|^2+\mathfrak{c}^2|p-q|^2\big]^{\frac{1}{2}}|p-q|^{b+|\varsigma|}}(\mathfrak{c}^2+|p-q|^2+|q|^2)^{\frac{\alpha}{2}}dp. 
    \end{aligned}
    \]
    We write $p\times q=q\times (q-p)$ and let $r=|p-q|$, $|q\times (q-p)|=|q| r \sin\theta$. Then
    \[
    \begin{aligned}
        I_1 & \le c(\mathfrak{c}+|q|)^{\frac{3|\varsigma|+a+2b}{2}}\int_{|q|}^{\infty}\int_{0}^{\pi} \frac{e^{-\frac{c\mathfrak{c}r}{T_\infty}}r^2\sin\theta}{\big[|q|^2r^2\sin^2\theta+\mathfrak{c}^2 r^2\big]^{\frac{1}{2}} r^{b+|\varsigma|}}(\mathfrak{c}^2+r^2)^{\frac{\alpha}{2}}d\theta dr\\
        &=c(\mathfrak{c}+|q|)^{\frac{3|\varsigma|+a+2b}{2}}\int_{|q|}^{\infty} \frac{e^{-\frac{c\mathfrak{c}r}{T_\infty}}(\mathfrak{c}^2+r^2)^{\frac{\alpha}{2}}}{r^{b+|\varsigma|-1}} \int_{0}^{\pi} \frac{\sin\theta}{\big[|q|^2\sin^2\theta+\mathfrak{c}^2\big]^{\frac{1}{2}} }d\theta dr\\
        &\le c(\mathfrak{c}+|q|)^{\frac{3|\varsigma|+a+2b}{2}} \cdot \frac{\arctan |q|}{|q|}
        \int_{|q|}^{\infty} \frac{e^{-\frac{c\mathfrak{c}r}{T_\infty}}(\mathfrak{c}^2+r^2)^{\frac{\alpha}{2}}}{r^{b+|\varsigma|-1}}  dr\\
        &\le c(\mathfrak{c}+|q|)^{\frac{3|\varsigma|+a+2b}{2}} \cdot \frac{1}{\mathfrak{c} + |q|} \cdot e^{-\frac{c\mathfrak{c}|q|}{T_\infty}} (\mathfrak{c}+|q|)^{\alpha + 1 - b - |\varsigma|} \\
        &\le c(\mathfrak{c}+|q|)^{-\eta+\alpha}.
    \end{aligned}
    \]
        
    For $I_2$, we have $0 \le |p-q|\le|p|+|q|<3|q|$. Then
    \[
    \begin{aligned}
        I_2 & \le c(\mathfrak{c}+|q|)^{\frac{3|\varsigma|+a+2b}{2}}\int_{|p|<2|q|} \frac{e^{-\frac{c\mathfrak{c}|p-q|}{T_\infty}}}{\big[|p\times q|^2+\mathfrak{c}^2|p-q|^2\big]^{\frac{1}{2}}|p-q|^{b+|\varsigma|}}(\mathfrak{c}^2+|p|^2)^{\frac{\alpha}{2}}dp \\
        &\le c(\mathfrak{c}+|q|)^{\frac{3|\varsigma|+a+2b}{2}+\alpha} \int_{0}^{3|q|} \int_{0}^{\pi} \frac{e^{-\frac{c\mathfrak{c}r}{T_\infty}}r^2\sin\theta}{\big[|q|^2r^2\sin^2\theta+\mathfrak{c}^2r^2\big]^{\frac{1}{2}} r^{b+|\varsigma|}}d\theta dr\\
        &\le c(\mathfrak{c}+|q|)^{\frac{3|\varsigma|+a+2b}{2}+\alpha} \int_{0}^{3|q|} \frac{e^{-\frac{c\mathfrak{c}r}{T_\infty}}}{r^{b+|\varsigma|-1}}\int_{0}^{\pi} \frac{\sin\theta}{\big[|q|^2\sin^2\theta+\mathfrak{c}^2\big]^{\frac{1}{2}} }d\theta dr \\
        &\le c(\mathfrak{c}+|q|)^{\frac{3|\varsigma|+a+2b}{2}+\alpha} \cdot \frac{\arctan |q|}{|q|} \cdot \int_{0}^{3|q|} \frac{e^{-\frac{c\mathfrak{c}r}{T_\infty}}}{r^{b+|\varsigma|-1}}dr. 
    \end{aligned}
    \]
    Note that for sufficiently large $|q|$, we have for $0<b+|\varsigma|<\frac{1}{2}$ that
    \[
        I_2 \le c(\mathfrak{c}+|q|)^{\frac{3|\varsigma|+a+2b}{2}-1+\alpha} = c(\mathfrak{c}+|q|)^{-\eta+\alpha}. 
    \]
    For small $|q|$, it can be easily derived that $I_2$ is bounded for $0 < b + |\varsigma| < \frac{1}{2}$ through the L'Hospital's rule. 
    Therefore (iii) holds for $\alpha > 0$. 
\end{proof}

\subsection{Properties of $P_0\hat{p}_1P_0$}

Recall from \eqref{eq:Null:L} that $\mathcal{N} = \text{span} \big\{ J_\infty^\frac{1}{2}\phi_i(p) \big\}$, $i=0,1,\cdots,4$. Let $\{\chi_i\}_{i=0}^4$ be the set of standard orthogonal bases of $\mathcal{N}$ satisfying $\langle \chi_i, \chi_j \rangle = \delta_{ij}$. In the case that $u_\infty = (u_{\infty,1},0,0)$, by Lemma \ref{lem:moments}, $\chi_i$ are in the form of 
\[
    \begin{aligned}
        \chi_0 & = \sqrt{\frac{\mathfrak{c}}{u_0}}
        J_\infty^{\frac{1}{2}} (p),\\
        \chi_1 & = \frac{1}{\sqrt{A_1} \mathfrak{c}}(p_1-a_1 u_{\infty,1})J_\infty^{\frac{1}{2}} (p),\\
        \chi_2 & = \sqrt{\frac{\mathfrak{c}}{a_1 u_0 T_\infty}} \cdot p_2 J_\infty^{\frac{1}{2}} (p),\\
        \chi_3 & = \sqrt{\frac{\mathfrak{c}}{a_1 u_0 T_\infty}} \cdot p_3 J_\infty^{\frac{1}{2}} (p),\\
        \chi_4 & = \frac{1}{\sqrt{A_1 A_2 T_\infty}} ( A_3 \mathfrak{c} + A_4 p_1 + A_1 p_0 ) J_\infty^{\frac{1}{2}} (p), 
    \end{aligned}
\]
where
\begin{align}
    u_0 & = \sqrt{\mathfrak{c}^2+u_{\infty,1}^2},\\
    z &:= \frac{\mathfrak{c}^2}{T_\infty},\\
    a_1(z) & := \frac{K_3(z)}{K_2(z)} > 0, \label{eq:a1} \\
    a_2(z) & := - a_1^2 + 6 \frac{a_1}{z} + 1 > 0, \label{eq:a2} \\ 
    a_3(z) & := - a_1^3 + \frac{6 a_1^2}{z} - \frac{6 a_1}{z^2} + a_1 - \frac{1}{z} > 0, \label{eq:a3} \\
    A_1(z, \frac{u_{\infty,1}}{\mathfrak{c}}) & := \frac{u_0}{\mathfrak{c}} ( a_2 \frac{u_{\infty,1}^2}{\mathfrak{c}^2} + \frac{a_1}{z} ) = \frac{u_0}{\mathfrak{c}} ( a_2 \frac{u_0^2}{\mathfrak{c}^2} - a_2 + \frac{a_1}{z} ), \\
    A_2(z, \frac{u_{\infty,1}}{\mathfrak{c}}) & := a_3 \frac{u_{\infty,1}^2}{\mathfrak{c}^2} + a_3 + \frac{a_2}{z} - \frac{a_1}{z^2} = a_3 \frac{u_0^2}{\mathfrak{c}^2} + \frac{a_2}{z} - \frac{a_1}{z^2}, \label{eq:A2} \\
    A_3(z, \frac{u_{\infty,1}}{\mathfrak{c}}) & := \frac{1}{z} \big( ( 2 a_2 - \frac{6 a_1}{z} - 1 ) \frac{u_{\infty,1}^2}{\mathfrak{c}^2} + a_2 - \frac{5 a_1}{z} - 1 \big) = \frac{1}{z} \big( (2 a_2 - \frac{6 a_1}{z} - 1 ) \frac{u_0^2}{\mathfrak{c}^2} - a_2 + \frac{a_1}{z} \big), \\ 
    A_4(z, \frac{u_{\infty,1}}{\mathfrak{c}}) & := -\frac{a_2 u_0^2 u_{\infty,1} }{\mathfrak{c}^3}. 
\end{align}

Note that $z$, $a_i$, $i=1,2,3$ and $A_i$, $i=1,\cdots,4$ are all dimensionless quantities, $a_i$ depends only on $z$, and $A_i$ depends on $z$ and $\frac{u_{\infty,1}}{\mathfrak{c}}$. 

We define the macroscopic operator $B$ as 
\begin{equation}\label{eq:B}   
    B := P_0\hat{p}_1P_0,
\end{equation}
which is the 5-dimensional linear bounded self-adjoint operator. Here $P_0$ is the mapping defined in (\ref{eq:mapping}). According to Lemma \ref{lem:moments}, its form is
\begin{equation}\label{eq:B:2}
    B = \begin{pmatrix}
        \mathfrak{c}\frac{u_{\infty,1}}{u_0}&\frac{T_\infty}{\sqrt{A_1\mathfrak{c} u_0}}&0&0&\frac{(a_1-z 
        a_2)u_{\infty,1}T_\infty^2}{\sqrt{A_1A_2\mathfrak{c}^5u_0T_\infty}}\\
        \frac{T_\infty}{\sqrt{A_1\mathfrak{c} u_0}}&\mathfrak{c}\frac{u_{\infty,1}}{u_0}&0&0&\frac{T_\infty^2}{\mathfrak{c} u_0\sqrt{T_\infty A_2}}\\
        0&0&\mathfrak{c}\frac{u_{\infty,1}}{u_0}&0&0\\
        0&0&0&\mathfrak{c}\frac{u_{\infty,1}}{u_0}&0\\
        \frac{(a_1-z a_2)u_{\infty,1}T_\infty^2}{\sqrt{A_1A_2\mathfrak{c}^5u_0T_\infty}}&\frac{T_\infty^2}{\mathfrak{c} u_0\sqrt{T_\infty A_2}}&0&0&\mathfrak{c}\frac{u_{\infty,1}}{u_0}+\frac{2u_{\infty,1}T_\infty^2(a_1-z a_2)}{\mathfrak{c}^3u_0A_2}
    \end{pmatrix}, 
\end{equation}
where the entries of $B$ are calculated by $\langle \chi_i, \hat{p}_1 \chi_j\rangle$, $i, j = 0, 1, \cdots,  4$.
The eigenvalues of the matrix $B$ can be readily calculated, and are given by
\begin{equation}\label{eq:lambda}
    \lambda_1 = \frac{a_3u_0u_{\infty,1}-\mathfrak{c}^2M}{\mathfrak{c} A_2},\quad \lambda_i = \mathfrak{c}\frac{u_{\infty,1}}{u_0} (i=2,3,4),\quad 
    \lambda_5 = \frac{a_3u_0u_{\infty,1}+\mathfrak{c}^2M}{\mathfrak{c} A_2},
\end{equation}
where
\begin{equation}\label{eq:M}
    M := \sqrt{\frac{1}{z}(a_2-\frac{a_1}{z})( a_3 + \frac{a_2}{z} - \frac{a_1}{z^2})}.
\end{equation}
\begin{lemma}\label{lem:Mach:number}
    {\rm (i)} $\displaystyle \lambda_1 = 0 \Leftrightarrow u_{\infty,1} = \sqrt{T_\infty \cdot \frac{z a_2(z) - a_1(z) }{z a_3(z)} }$.
    
    {\rm (ii)} $\displaystyle \lambda_5 = 0 \Leftrightarrow u_{\infty,1} = - \sqrt{T_\infty \cdot \frac{z a_2(z) - a_1(z) }{z a_3(z)} }$.
\end{lemma}
\begin{proof}
    We only need to prove (i). Since $\lambda_1 = 0$, then 
    \[
        a_3^2 (u_{\infty,1}^2+\mathfrak{c}^2) u_{\infty,1}^2 = \mathfrak{c}^4 M^2. 
    \]
    Hence, 
    \[
        \Big( \frac{u_{\infty,1}}{\mathfrak{c}} \Big)^4 + \Big( \frac{u_{\infty,1}}{\mathfrak{c}} \Big)^2 = \frac{M^2}{a_3^2} = \frac{(a_2-\frac{a1}{z})(a_3+\frac{a_2}{z} - \frac{a_1}{z^2})}{z a_3^2} = \Big( \frac{a_2-\frac{a1}{z}}{z a_3} \Big)^2 + \frac{a_2-\frac{a1}{z}}{z a_3}. 
    \]
    Therefore, 
    \[
        u_{\infty,1} = \mathfrak{c} \sqrt{\frac{a_2-\frac{a1}{z}}{z a_3}} = \sqrt{T_\infty \cdot \frac{z a_2(z) - a_1(z) }{z a_3(z)} }. 
    \]
    The lemma is proved. 
\end{proof}

Note that the above lemma leads to the definitions of the sound speed $c_\infty$ in \eqref{eq:Mach:number}. 

\begin{remark}\label{rmk:Mach:number}
    There is no eigenvalue of $B$ vanishes when $\mathcal{M}^\infty = \frac{u_{\infty,1}}{c_\infty}$ does not equal $\pm 1$ or $0$. 
\end{remark}

Let 
\[
    I^+ := \{ j \mid \lambda_j>0 \}, \qquad I^- := \{ j \mid \lambda_j<0 \}. 
\]
The operator $B$ on space $\mathcal{N}$ is decomposed into positive part $B_+$ and negative part $B_-$, and its corresponding eigen-mappings are $P_0^+$ and $P_0^-$, respectively. If $\mathcal{M}^\infty\ne 0,\pm 1$, then there are no zero eigenvalues of $B$. Therefore
\begin{equation}\label{eq:contruct:damping}
    B = B_+ + B_-, \qquad P_0 = P_0^+ + P_0^-.
\end{equation}
It will be demonstrated that the number of solvability conditions for problem \eqref{eq:BVP} is determined by  $\mathcal{M}^\infty$. We will show that the codimension of the permissible boundary data changes with the value of $\mathcal{M}^\infty$. 

\section{Existence for linear problem with damping}\label{sec:elp}

In this section, we will prove the existence of solution for the linearized relativistic Boltzmann equation with damping by energy method. 

We consider the following boundary value problem with the far field state $J_\infty$. We shall first look for the solution of (\ref{eq:BVP}) in the form 
\[
    F(x,p)=J_\infty(p)+W_0(p)f(x,p),
\]
where $W_0(p)$ is the weight in (\ref{eq:weight}) with $\beta=0$. Then problem (\ref{eq:BVP}) reduces to
\begin{equation}\label{eq:linearized:F}
    \left\{ 
    \begin{array}{ll}
        \hat{p}_1 \partial_x f = L(f) + \Gamma(f,f),\quad & x\in \mathbb{R}^+, p \in \mathbb{R}^3,\\
        f|_{x=0}=a_0( p ),  & p_1>0,( p_2,p_3 ) \in \mathbb{R}^2,\\
        \displaystyle \lim_{x\to +\infty} f = 0,  & p\in \mathbb{R}^3,
    \end{array} 
    \right.
\end{equation}
where
\[
    a_0 = W_0^{-1}(F_0 - J_\infty).
\]

Add a damping term $-\gamma P_0^+ \hat{p}_1 f$, with $\gamma>0$ and $P_0^+$ defined in (\ref{eq:contruct:damping}), to (\ref{eq:linearized:F}) and (\ref{eq:BVP}), then rewrite the equations as
\begin{equation}\label{eq:f}
    \left\{ 
    \begin{array}{ll}
        \hat{p}_1 \partial_x f = L(f) + \Gamma(f,f)-\gamma P_0^+ \hat{p}_1 f,\quad & x\in \mathbb{R}^+, p \in \mathbb{R}^3,\\
        f|_{x=0}=a_0( p ),  & p_1>0,( p_2,p_3 ) \in \mathbb{R}^2,\\
        \displaystyle \lim_{x\to +\infty} f = 0,  & p\in \mathbb{R}^3,
    \end{array} 
    \right.
\end{equation}
and
\begin{equation}\label{eq:F:damping}
    \left\{ 
    \begin{array}{ll}
        \hat{p}_1 \partial_x F = Q( F, F )-\gamma W_0 P_0^+ \hat{p}_1 W_0^{-1}(F-J_\infty), \quad & x\in \mathbb{R}^+, \quad p\in \mathbb{R}^3,\\
        F|_{x=0} = F_0( p ), & p_1>0, \quad ( p_2, p_3 ) \in \mathbb{R}^2,\\
        \displaystyle \lim_{x\to +\infty} F = J_{\infty}( p ),  & p\in \mathbb{R}^3. 
    \end{array} \right.
\end{equation}

Let $f=e^{-\tau x}h$ for $\tau>0$, then \eqref{eq:f} becomes 
\begin{equation}\label{eq:h:2}
    \left\{ 
    \begin{array}{ll}
        \hat{p}_1 \partial_x h - \tau \hat{p}_1 h -L(h)=\zeta-\gamma P_0^+ \hat{p}_1 h, \quad & x\in \mathbb{R}^+,p\in \mathbb{R}^3,\\
        h|_{x=0}=a_0 ( p ),  & p_1>0,( p_2,p_3 ) \in \mathbb{R}^2,\\
        \displaystyle \lim_{x\to +\infty} h = 0, & p\in \mathbb{R}^3,
    \end{array} 
    \right.
\end{equation}
where 
\begin{equation}\label{eq:zeta}
   \zeta := e^{-\tau x}\Gamma(h,h). 
\end{equation}

The problem we consider in this section is to view the nonlinear term $e^{-\tau x}\Gamma(h,h)$ by a given function $\zeta(x,p)$ where $\zeta$ satisfies
\[
    P_0 \zeta =0,\qquad \| \zeta \| < \infty.
\]

The idea of the proof is to introduce a linear functional on a subspace of $L^2_{x,p}$, then demonstrate that the aforementioned linear functional is bounded through the energy estimates, and finally obtain the existence of the solution through the Riesz representation theorem. 

We define the test function space by
\[
    \mathcal{V} := \lbrace \varphi (x,p)\in C_0^\infty([0,\infty)\times\mathbb{R}^3) \mid \varphi^0 := \varphi|_{x=0}=0 \text{ for }  p_1 < 0 \rbrace, 
\]
and let
\[
    \mathcal{U} := \lbrace \xi \mid \xi = - \hat{p}_1 \partial_x \varphi - \tau \hat{p}_1 \varphi + \gamma \hat{p}_1 P_0^{+} \varphi - L(\varphi), \varphi \in \mathcal{V} \rbrace. 
\]

A linear functional on $\mathcal{U}$ is defined as
\[
    \mathcal{L}(\xi) := (\zeta,\varphi) + \langle \hat{p}_1 a_0,\varphi^0 \rangle_+, \qquad \xi\in \mathcal{U}. 
\]

By testing \eqref{eq:h:2} with $\varphi$, it is easy to see that if $h$ is a solution of \eqref{eq:h:2}, then for any $\xi \in \mathcal{U}$, it holds that
\[
    (h,\xi) = \mathcal{L}(\xi).
\]

Based on the above definitions, we can obtain the following lemma about the linear operator $\mathcal{L}$.
\begin{lemma}\label{lem:L:bdd}
    If $\gamma > \tau > 0$ are sufficiently small, then the linear operator $\mathcal{L}$ is bounded on $\mathcal{U}$, and satisfies
    \[
        |\mathcal{L}(\xi)|\le c(\|\zeta\|+\langle \hat{p}_1 a_0,a_0 \rangle_+^{\frac{1}{2}} )\|\xi\|. 
    \]
    In addition, there exists an extension $\bar {\mathcal{L}}$ of $\mathcal{L}$ defined in $L_{x,p}^2$, such that
    \[
        |\bar {\mathcal{L}}| \le c\|\xi\|, 
    \]
    for all $\xi \in L_{x,p}^2$, and with the bound unchanged.	
\end{lemma}
\begin{proof}
    For any $\xi \in \mathcal{U}$, there exists a $\varphi \in  \mathcal{V}$ such that $\xi = - \hat{p}_1 \partial_x \varphi-\tau \hat{p}_1 \varphi +\gamma \hat{p}_1 P_0^{+} \varphi - L(\varphi)$. 
    We take the $L_{x,p}^2$ inner product of any $\xi\in \mathcal{U}$ and the corresponding $\varphi \in  \mathcal{V}$, 
    \[
        (\xi,\varphi)=(-\hat{p}_1 \varphi_x,\varphi) - \tau(\hat{p}_1\varphi,\varphi) + \gamma(\hat{p}_1 P_0^+ \varphi,\varphi) - (L \varphi,\varphi). 
    \]
    
    The first term is
    \[
    \begin{aligned}
        ( - \hat{p}_1 \varphi_x, \varphi) & = - \int_{\mathbb{R}^3} \int_{0}^{\infty} \hat{p}_1 \varphi_x \varphi dx dp \\
        & = - \frac{1}{2}\int_{\mathbb{R}^3} \hat{p}_1 \varphi ^2 |_{x=0}^{x=\infty}dp \\
        &=\frac{1}{2}\int_{p_1>0} \hat{p}_1 (\varphi^0) ^2 dp+\frac{1}{2}\int_{p_1<0} \hat{p}_1 (\varphi^0) ^2 dp\\
        &=\frac{1}{2}\langle \hat{p}_1 \varphi^0, \varphi^0 \rangle_+. 
    \end{aligned}
    \]
    
    For convenience, we let $\varphi_i := P_i\varphi$, $i = 0, 1$. For the second and third terms, we have for some generic constant $c$ that
    \begin{equation}\label{eq:estimate}
    \begin{aligned}
        & -\tau ( \hat{p}_1\varphi,\varphi )+\gamma(P_0^+ \hat{p}_1 \varphi,\varphi)\\
        = & - \tau \big( ( P_0 \hat{p}_1 \varphi_0, \varphi_0 ) + 2 ( P_0 \hat{p}_1 \varphi_1, \varphi_0 ) +( \hat{p}_1 \varphi_1, \varphi_1) \big) + \gamma \big( ( P_0^+ \hat{p}_1 \varphi_0, \varphi_0 ) + ( P_0^+ \hat{p}_1 \varphi_1, \varphi_0) \big) \\
        = & - \tau (P_0^-\hat{p}_1\varphi_0,\varphi_0)+(\gamma-\tau)(P_0^+\hat{p}_1\varphi_0,\varphi_0)-2\tau(P_0\hat{p}_1\varphi_1,\varphi_0)-\tau(\hat{p}_1\varphi_1, \varphi_1)+\gamma(P_0^+\hat{p}_1\varphi_1,\varphi_0)\\
        \ge & 2 c \tau \| \varphi_0 \|^2 - c^{-1} \tau \|\varphi_1\|\ \|\varphi_0\|-\tau(\hat{p}_1\varphi_1,\varphi_1)\\
        \ge & c \tau \| \varphi_0 \|^2 - c^{-3} \tau\|\varphi_1\|^2-\tau(\hat{p}_1\varphi_1,\varphi_1), 
    \end{aligned}
    \end{equation}
    where we assume that $\gamma = O(1) \tau$ for some positive constant $O(1) > 2$. 

    Since $(L\varphi, \varphi) = (L\varphi_1,\varphi_1 )$, by \eqref{eq:estimate}, (i) and (iii) of Proposition \ref{prop:L}, we know that for sufficiently small $\tau > 0$, 
    \[
    \begin{aligned}
        (\xi,\varphi) & \ge \frac{1}{2} \langle \hat{p}_1 \varphi^0, \varphi^0 \rangle_+ + c \tau  \| \varphi_0 \|^2 - c \tau \| \varphi_1 \|^2 - \tau ( \hat{p}_1 \varphi_1, \varphi_1 ) + c_0 ( \varphi_1, \nu(p) \varphi_1 ) \\
        & \ge \frac{1}{2}\langle \hat{p}_1 \varphi^0, \varphi^0 \rangle_+ + c \tau \| \varphi_0\|^2 - c \tau \|\varphi_1\|^2-\tau(\hat{p}_1\varphi_1,\varphi_1)+c_0(\varphi_1,p_0^\frac{a}{2}\varphi_1)\\
        &\ge \frac{1}{2}\langle \hat{p}_1 \varphi^0, \varphi^0 \rangle_+ +c\tau\|\varphi_0\|^2+\frac{c_0}{2}(\varphi_1,p_0^\frac{a}{2}\varphi_1)\\
        &\ge c(\langle \hat{p}_1 \varphi^0, \varphi^0 \rangle_+ +\|\varphi\|^2). 
    \end{aligned}
    \]
    Then
    \[
        \langle \hat{p}_1 \varphi^0, \varphi^0 \rangle_+^{\frac{1}{2}} + \| \varphi \| \le c \| \xi \|. 
    \]
    Thus we have 
    \begin{equation}\label{eq:L}
    \begin{aligned}
        |\mathcal{L}(\xi)|&\le  |(\zeta,\varphi)| + | \langle \hat{p}_1 a_0,\varphi^0 \rangle _+| \\
        &\le \|\zeta\| \cdot \|\varphi\| + \langle \hat{p}_1 a_0,a_0\rangle _+^{\frac{1}{2}} \cdot \langle \hat{p}_1 \varphi^0, \varphi^0\rangle _+^{\frac{1}{2}}\\
        &\le c(\|\zeta\|+\langle \hat{p}_1 a_0,a_0\rangle _+^{\frac{1}{2}})(\|\varphi\|+\langle \hat{p}_1 \varphi^0, \varphi^0 \rangle _+^{\frac{1}{2}})\\
        &\le c(\|\zeta\| + \langle \hat{p}_1 a_0,a_0 \rangle_+^{\frac{1}{2}} ) \| \xi \|. 
    \end{aligned}
    \end{equation}
    Thus it can be concluded that $\mathcal{L}$ is bounded on $\mathcal{U}$, with the bound depending on the boundary data $a_0$ and the given function $\zeta$. 
    In accordance with the Hahn-Banach theorem, there exists an extension $\bar{\mathcal{L}}$ of $\mathcal{L}$ to the space $L^2_{x,p}$ such that for any $\xi\in L^2_{x,p}$, we have that
    \[
        |\bar {\mathcal{L}}(\xi)| \le c\|\xi\|, 
    \]
    with the same bound in \eqref{eq:L}. 
\end{proof}

We are ready to obtain the following existence theorem for the linearized equation (\ref{eq:h:2}) with damping and a given function $\zeta(x,p)$. 
\begin{theorem}\label{thm:lin:existence}
    Suppose $\gamma,\tau$ are suitably small such that Lemma \ref{lem:L:bdd} holds. If the boundary condition $a_0$ and the source term $\zeta$ satisfy
    \[
        \|\zeta\|+\langle \hat{p}_1 a_0, a_0\rangle_+^\frac{1}{2}<\infty,
    \]
    then there exists a unique solution $h\in L^2_{x,p}$ of \eqref{eq:h:2} satisfying
    \begin{equation}\label{eq:L2:estimate}
        \langle |\hat{p}_1| h^0, h^0 \rangle_- + (\nu h,h) + \| \nu^{-\frac{1}{2}} |\hat{p}_1| h_x \|^2 \le c(\|\zeta\|^2+\langle \hat{p}_1 a_0, a_0\rangle_+).
    \end{equation}
    where $h^0 := h\mid_{x=0}$. 
\end{theorem}
\begin{proof}
    For the above functional $\bar {\mathcal{L}}$ on $L^2$, we can apply the Riesz representation theorem and it implies that there exists a unique $h\in L^2_{x,p}$ such that
    \begin{equation}\label{eq:L:tilde}
        \bar {\mathcal{L}}(\xi)=(h,\xi).
    \end{equation}
    This indicates that $h$ is a weak solution to the linear equation \eqref{eq:h:2}. 
        
    Now our objective is to demonstrate that both $h_x$ for almost all $(x, p )$ and the trace of $h$ at $x = 0$ are well-defined. In order to achieve this, a specific family of test functions $\phi$ was chosen for analysis as follows:
    \[
        \phi = \int_{x}^{\infty}\eta(x',p)dx',  
    \]
    where $\eta \in C_0^\infty([0,\infty)\times\mathbb{R}^3)$ satisfies $\int_{0}^{\infty}\eta (x',p)dx' = 0$ for any $p$. Applying this test function to (\ref{eq:h:2}) yields 
    \[
        \Big( \int_{0}^{x}(-\tau\hat{p}_1 h+\gamma P_0^+\hat{p}_1 h-Lh-\zeta)dx'+\hat{p}_1 h, \eta \Big) = 0. 
    \]
    By the choice of our test function $\eta$, we have that
    \[
        \int_{0}^{x}(-\tau\hat{p}_1 h+\gamma P_0^+\hat{p}_1 h-Lh-\zeta)dx'+\hat{p}_1 h = b(p) 
    \]
    holds for almost all $(x, p )$, where $b(p )$ is a function of $p$ only. Therefore, $h_x$ is well-defined for almost all $(x,p)$ and the trace of $h$ at $x = 0$ is well-defined. Hence we have (\ref{eq:h:2}) for almost all $(x, p )$.
        
    By \eqref{eq:L} and \eqref{eq:L:tilde}, we have 
    \begin{equation}\label{eq:h:est}
        \|h\|\le c(\|\zeta\|+\langle \hat{p}_1 a_0, a_0\rangle_+^\frac{1}{2}),
    \end{equation}
    which implies the uniqueness of $h$. So far we know that $h\in L^2_{x,p}$. Since $|\hat{p}_1|$ is bounded, $|\hat{p}_1|^{\frac{1}{2}}h\in L^2_{x,p}$. 
        
    Now we claim that the solution obtained indeed satisfies the given boundary condition. Actually, if we take the inner product of (\ref{eq:h:2}) with $\varphi \in \mathcal{V}$,  we have 
    \[
        (\hat{p}_1 h_x-\tau\hat{p}_1 h-Lh + \gamma P_0^+\hat{p}_1 h,\varphi) = (\zeta,\varphi). 
    \]
    After integration by parts, we have that
    \[
        (h,-\hat{p}_1 \varphi_x-\tau\hat{p}_1 \varphi-L\varphi + \gamma \hat{p}_1 P_0^+ \varphi)=(\zeta,\varphi)+\langle \hat{p}_1 h^0,\varphi^0 \rangle_+. 
    \]
        
    In accordance with the definition of $\xi$ and the premise that $h$ satisfies (\ref{eq:L:tilde}), it holds that
    \[
        (\zeta,\varphi)+\langle \hat{p}_1 h^0,\varphi^0 \rangle_+ =(\xi,h)=\mathcal{L}(\xi)=(\zeta,\varphi)+\langle \hat{p}_1 a_0,\varphi^0 \rangle_+, 
    \]
    which implies 
    \[
        h^0=a_0, \text{ for a.e. } p_1 > 0. 
    \]
    The boundary condition holds. 
        
    Finally, we prove \eqref{eq:L2:estimate}. 
    Applying $h$ to (\ref{eq:h:2}) and then integrating over $(x, p)$ yields that
    \begin{equation}\label{eq:temp}
        (\hat{p}_1 h_x, h) - \tau (\hat{p}_1 h, h) + \gamma ( P_0^+ \hat{p}_1 h, h) - ( L h,h) = ( \zeta, h ).
    \end{equation}
    The first item is 
    \begin{equation}\label{eq:temp:1st:term}
    \begin{aligned}
        (\hat{p}_1 h_x, h) 
        & = \frac{1}{2}\int_{\mathbb{R}^3} \hat{p}_1 h^2|_{x=0}^{x=\infty }dp\\
        & = - \frac{1}{2} \langle \hat{p}_1 h^0, h^0 \rangle \\
        & = - \frac{1}{2} \langle \hat{p}_1 a_0, a_0 \rangle_+ +\frac{1}{2} \langle |\hat{p}_1| h^0, h^0 \rangle_-. 
    \end{aligned}
    \end{equation}
    Substitute \eqref{eq:temp:1st:term} into \eqref{eq:temp}, by $L=-\nu+K$, we have that
    \begin{equation}\label{eq:lem5:temp}
    \begin{aligned}
        \frac{1}{2} \langle |\hat{p}_1| h^0, h^0 \rangle_- - \tau ( \hat{p}_1 h, h) + \gamma (P_0^+\hat{p}_1 h, h) + (\nu h,h) & = (Kh,h) + (\zeta,h) + \frac{1}{2} \langle \hat{p}_1 a_0, a_0 \rangle_+  \\
        & \le c \| h \|^2 + \|\zeta\| \|h\| + \frac{1}{2} \langle \hat{p}_1 a_0, a_0 \rangle_+ \\
        &\le c ( \|\zeta\|^2 + \|h\|^2 + \langle \hat{p}_1 a_0, a_0 \rangle_+ ). 
    \end{aligned}
    \end{equation}

    Notice that 
    \[
        | ( \hat{p}_1 h, h) | \le \|h\|^2, 
    \]
    \[
        | ( P_0^+\hat{p}_1 h, h) | = | (\hat{p}_1 h, P_0^+ h ) | \le \| \hat{p}_1 h \| \cdot \| h \| \le \| h \|^2.
    \] 
    We know from the above estimate, \eqref{eq:lem5:temp} and \eqref{eq:h:est} that 
    \begin{equation}\label{eq:nu:h}
        \langle | \hat{p}_1 | h^0, h^0 \rangle_- + ( \nu h, h) \le c(\|\zeta\|^2 + \langle \hat{p}_1 a_0, a_0 \rangle_+).
    \end{equation}

    By \eqref{eq:h:2}, \eqref{eq:nu:h}, we have that 
    \[
        \begin{aligned}
            \|\nu^{-\frac{1}{2}}|\hat{p}_1| h_x\|^2 & \le \|\nu^{-\frac{1}{2}}\tau \hat{p}_1h\|^2 + \|\nu^{-\frac{1}{2}}Lh\|^2 + \|\nu^{-\frac{1}{2}}\zeta\|^2 + \|\nu^{-\frac{1}{2}}\gamma P_0^+ \hat{p}_1 h\|^2\\
            & \le c (\|h\|^2+\|\zeta\|^2+\|\nu^{-\frac{1}{2}}\nu h\|^2+\|Kh\|^2 +  \|\hat{p}_1 h\|^2)\\
            & \le c (\|h\|^2+\|\zeta\|^2+\|\nu^{\frac{1}{2}}h\|^2)\\
            & \le c(\|\zeta\|^2 + \langle \hat{p}_1 a_0, a_0 \rangle_+).
        \end{aligned}    
    \]
    Therefore the unique solution $h$ satisfies \eqref{eq:L2:estimate}. 
\end{proof}

\section{Weighted estimates for the linearized problem with damping}\label{sec:we}

In order to obtain the existence of solutions to the nonlinear problem, we need to estimate $\|\cdot\|_{L^\infty_{x,p,\beta}} $ norm of the solution for the linear problem. This section is devoted to achieve the estimate. 

\begin{lemma}\label{lem:estimate:f}
    Assume $f(x)\in L^2([0,\infty))$ and $f \rightarrow 0$ as $x \rightarrow \infty$, the derivative $f_x$ is defined for almost all $x$ and $f_x\in L^2([0,\infty))$. Then $f\in L^\infty([0,\infty))$ and
    \[
        \|f\|^2_{L^\infty([0,\infty))}\le c\|f\|_{L^2([0,\infty))} \ \|f_x\|_{L^2([0,\infty))},
    \]
    where $c$ is a constant independent of $f$. 
\end{lemma}
\begin{proof}
    Direct computation gives that for almost every $y>0$, 
    \[
        |f|^2(y) = | \int_y^\infty(f^2)_x dx | = 2 \Big| \int_y^\infty f f_x dx \Big| \le c\|f\|_{L^2 ( [0, \infty) ) } \|f_x\|_{L^2([0,\infty) )}. 
    \]
    The lemma holds. 
\end{proof}

We rewrite equation \eqref{eq:h:2} as
\[
    \left\{ 
    \begin{aligned}
        & h_x= \tau h+\frac{1}{\hat{p}_1}Lh+\frac{1}{\hat{p}_1}\zeta-\frac{\gamma}{\hat{p}_1}P_0^+\hat{p}_1 h = (\tau-\frac{\nu(p)}{\hat{p}_1})h+\frac{1}{\hat{p}_1}(\bar{K}h+\zeta),\\
        & h|_{x=0}= a_0 ( p ),  \qquad p_1>0,( p_2,p_3 ) \in \mathbb{R}^2,\\
        \displaystyle & \lim_{x\to +\infty} h = 0, \qquad p\in \mathbb{R}^3,
    \end{aligned} 
    \right.
\]
where
\[
    \bar{K} := K - \gamma P_0^+ \hat{p}_1.
\]
By Proposition \ref{prop:L}, $K$ is a compact operator and is bounded from $L^\infty_{p,\beta}$ to $L^\infty_{p,\beta+\eta}$ for all $\beta \ge0$ and $\eta\in(0,1]$ given by \eqref{eq:eta}, and bounded from $L^2_p$ to $L^\infty_p$. Since the operator $P_0^+ \hat{p}_1$ is a mapping onto a subspace of $\mathcal{N}$, which is of finite dimensions, the operator $\bar{K}$ has the same properties as $K$. 

Let
\[
    \kappa(x,p) := ( - \tau + \frac{\nu(p)}{\hat{p}_1})x.
\]
It is obvious that $\kappa(x, p ) > 0$ for $x \hat{p}_1 > 0$ when $\tau > 0$ is sufficiently small. Therefore, the solution $h$ can be formally written as follows:
\begin{equation}\label{def:sol:h}
    h=\tilde{a}+U(\bar{K}h+\zeta), 
\end{equation}
where 
\[
    \tilde{a} := 
    \left\{ 
    \begin{array}{ll}
        e^{-\kappa(x,p)}a_0(p), \quad & \hat{p}_1 >0,\\
        0,\quad & \hat{p}_1 < 0, 
    \end{array} 
    \right. 
\]
\[
    U(\zeta) := 
    \left\{ 
    \begin{array}{ll}
        \int_{0}^{x}e^{-\kappa(x-x',p)}\frac{1}{\hat{p}_1}\zeta(x',p)dx', \quad & \hat{p}_1 >0,\\
        - \int_{x}^{\infty}e^{-\kappa(x-x',p)}\frac{1}{\hat{p}_1}\zeta(x',p)dx',\quad & \hat{p}_1 < 0.
    \end{array} 
    \right. 
\]

\begin{lemma}\label{lem:U}
    The operator $U$ has the following properties:
    \[
    \begin{aligned}
        &\| U(\zeta) \|_{L^q_{x}}\le c\nu(p)^{-1}\| \zeta \|_{L^q_{x}}, \\
        &\| U(\zeta) \|_{L^r_{p,\beta}(L^q_{x})}\le c\|\nu(p)^{-1} \zeta \|_{L^r_{p,\beta }(L^q_{x})}, 
    \end{aligned} 
    \]
    where $1 \le q,r \le \infty$, $\beta\in\mathbb{R}$. 
\end{lemma}

\begin{proof}
    We rewrite $U(\zeta)$ as 
    \[
        U(\zeta) = \int_{I_{\hat{p}_1}}S(x-x',p) \zeta(x',p) dx',
    \]
    where
    \[
        I_{\hat{p}_1} := 
            \left\{ 
            \begin{aligned}
                & [ 0, x ],  \quad && \hat{p}_1 > 0, \\
                & [ x, \infty], && \hat{p}_1 < 0, 
            \end{aligned} 
            \right. 
    \]
    and 
    \[
        S(x'',p) := 
            \left\{ 
            \begin{array}{ll}
                e^{(\tau-\frac{\nu(p)}{\hat{p}_1})x''}\frac{1}{\hat{p}_1},  \quad & \hat{p}_1 >0, \\
                -e^{(\tau-\frac{\nu(p)}{\hat{p}_1})x''}\frac{1}{\hat{p}_1},   & \hat{p}_1 < 0.
            \end{array} 
            \right. 
    \]
    Then we have
    \[
            \int_{I_{\hat{p}_1}}|S(x-x',p)|dx' =
            \left\{ 
            \begin{array}{ll}
                \int_{0}^{x} e^{ ( \tau - \frac{\nu(p) }{ \hat{p}_1} ) x'' } \frac{1}{\hat{p}_1}dx''\le \frac{1}{\nu(p)-\tau\hat{p}_1}, \quad & \hat{p}_1 >0,\\
                \int_{0}^{\infty} e^{ ( - \tau + \frac{\nu(p) }{ \hat{p}_1} ) x'' } \frac{1}{|\hat{p}_1|}dx''= \frac{1}{\nu(p)+\tau|\hat{p}_1|}, \quad & \hat{p}_1 < 0. 
            \end{array} 
            \right.
    \]
    Therefore, for sufficiently small $\tau>0$, it holds that
    \[
        \int_{I_{\hat{p}_1}} | S(x-x',p) | dx' \le \frac{c}{\nu(p)}. 
    \]

    Now we estimate $U (\zeta)$. For $q=\infty$,
    \begin{equation}\label{eq:U:infty}
        \begin{aligned}
            |U(\zeta)(x,p)| & \le \int_{I_{\hat{p}_1}}|S(x-x',p)| \cdot | \zeta(x',p)|dx'  \\     
            & \le \underset{x\in \mathbb{R}^+}{\sup}|\zeta(x,p)|\int_{I_{\hat{p}_1}}|S(x-x',p)| dx'\\
            & \le \frac{c}{\nu(p)} \|\zeta\|_{L_x^\infty}.
        \end{aligned} 
    \end{equation}
    For $1 \le q< \infty$ and $\frac{1}{q}+\frac{1}{q'}=1$, we have that
    \[
        |U(\zeta)(x,p)| \le \Big( \int_{I_{\hat{p}_1}}|S(x-x',p)|dx' \Big)^\frac{1}{q'} \Big( \int_{I_{\hat{p}_1}}|S(x-x',p)| \cdot | \zeta(x',p)|^q dx' \Big) ^\frac{1}{q}. 
    \]
    Then it holds that
    \begin{equation}\label{eq:U:le:infty}
        \begin{aligned}
            \int_{0}^{\infty} | U(\zeta(x,p)) |^q dx & \le c \nu(p)^{-\frac{q}{q'}} \int_{0}^{\infty} \int_{I_{\hat{p}_1}} | S(x-x',p) | \cdot |\zeta(x',p)|^q dx' dx \\
            & \le c \nu(p)^{-\frac{q}{q'}} \int_{0}^{\infty} \int_{I_{\hat{p}_1}'} | S(x-x',p) | \cdot | \zeta(x',p) |^q dx dx' \\
            & \le c \nu(p)^{-\frac{q}{q'}-1}\int_{0}^{\infty}|\zeta(x',p)|^qdx', 
        \end{aligned} 
    \end{equation}
    where
    \[
        I_{\hat{p}_1}' := 
            \left\{ 
            \begin{array}{ll}
                [ x', \infty ),  \quad & \hat{p}_1 >0, \\
                ( 0, x' ],  & \hat{p}_1 < 0. 
            \end{array} 
            \right. 
    \]

    In the derivation of \eqref{eq:U:le:infty}, we used the estimate
    \[
            \int_{I_{\hat{p}_1}'}|S(x-x',p)|dx =
            \left\{ 
            \begin{array}{ll}
                \int_{0}^{\infty} e^{ ( \tau - \frac{\nu(p) }{ \hat{p}_1} ) x'' } \frac{1}{\hat{p}_1}dx''= \frac{1}{\nu(p)-\tau\hat{p}_1} , \quad & \hat{p}_1 >0,\\
                 \int_{-x'}^{0} e^{ ( \tau - \frac{\nu(p) }{ \hat{p}_1} ) x'' } \frac{1}{|\hat{p}_1|}dx''\le \frac{1}{\nu(p)+\tau|\hat{p}_1|}, \quad & \hat{p}_1 < 0, 
            \end{array} 
            \right.
    \]
    which implies for sufficiently small $\tau>0$ that
    \[
        \int_{I_{\hat{p}_1}'} | S(x-x',p) | dx \le \frac{c}{\nu(p)}. 
    \]

    By \eqref{eq:U:infty} and \eqref{eq:U:le:infty}, it yields the first estimate of the lemma. The second estimate is a direct consequence of the first one. 
\end{proof}

Set
\begin{equation}\label{eq:w}
    w_{-\varrho,\alpha} := \left\{ 
        \begin{array}{ll}
        | p_1 |^{-\varrho}p_0^{\alpha}, \quad & |p_1| <1,\\
        p_0^{\alpha},\quad & |p_1| \ge 1. 
        \end{array} 
    \right. 
\end{equation}
\begin{lemma} \label{lem:bar:a:bounded}
    For $0<\varrho<1$, it holds for some constant $c$ depending only on $\varrho$ that 
    \[
        \int_{-\infty}^{+\infty} \frac{ |s|^{-\varrho} }{|s-a| + 1} ds<c.
    \]
\end{lemma}
\begin{proof}
    Without loss of generality, we assume that $a>0$. For $a>0$ small, say $0< a < 1/2$, it is easy to obtain the result.
    For large $|a|$, we separate the integral region $(-\infty,\infty)$ as
    \[
    \begin{aligned}
        \int_{-\infty}^{+\infty} \frac{ |s|^{-\varrho} }{|s-a| + 1} ds = \int_{-\infty}^{0} + \int_{0}^{\frac{a}{2}} +\int_{\frac{a}{2}}^{a}+\int_{a}^{2a}+\int_{2a}^{+\infty} =: R_1+R_2+R_3+R_4+R_5.
    \end{aligned}  
    \]
    Obviously, $R_1$ is bounded since $0<\varrho<1$, and it holds as $a \to +\infty$ that
    \[
    \begin{aligned}
        R_2 & \le c\int_{0}^{\frac{a}{2}} \frac{ |s|^{-\varrho} }{ \frac{a}{2} + 1} ds \le c\frac{a^{1-\varrho}}{a+1} \rightarrow 0,\\
        R_3 & \le c\int_{\frac{a}{2}}^{a} \frac{ a^{-\varrho} }{a-s+1} ds = c a^{-\varrho}\ln(\frac{a}{2}+1)\rightarrow 0,\\
        R_4 & \le c\int_{a}^{2a} \frac{ a^{-\varrho} }{s-a+1} ds = c a^{-\varrho}\ln(a+1)\rightarrow 0,\\
        R_5 & \le c\int_{2a}^{+\infty} \frac{1}{(s-a+1)^{1+\varrho}} ds = c (a+1)^{-\varrho}\rightarrow 0.
    \end{aligned}  
    \]
    The lemma is proved.
\end{proof}

\begin{lemma}\label{lem:k:weight}
    For any $0<\varrho<1$ and $0\le\alpha\le\eta$, it holds that
    \begin{equation}
        \int |k(p,q)|w_{-\varrho,\alpha}dp \le c,
    \end{equation}
    where $\eta$ is defined in \eqref{eq:eta}.
\end{lemma}
\begin{proof}
    Over the region $|p_1|\ge 1$, we have from (iv) in Proposition \ref{prop:k:weight} that
    \[
    \begin{aligned}
        \int_{|p_1|\ge 1} |k(p,q)|w_{-\varrho,\alpha}dp \le \int |k(p,q)| ( \mathfrak{c}^2 + |p|^2)^{\frac{\alpha}{2}} dp \le c ( \mathfrak{c}^2 + |q|^2)^{\frac{\alpha-\eta}{2}} \le c, 
    \end{aligned}
    \]
    if $\alpha \le \eta$.
    
    It only remains to compute the integral over $|p_1|<1$. We separate the cases based on the value of $q$. Let $p' := (p_2,p_3)$, $q' := (q_2,q_3)$, then $|p|^2 = p_1^2 + |p'|^2$ and $|q|^2 = q_1^2 + |q'|^2$.
    
    {\it Case 1: $q=0$. } From \eqref{eq:k:es} and a scaling transform $s := p_1/|p'|$, we have that
    \[
    \begin{aligned}
        \int_{|p_1|<1} |k(p,q)|w_{-\varrho,\alpha}dp \le & c \int_{|p_1|<1} \frac{(\mathfrak{c}+|p|)^{1-\eta+\alpha} e^{-\frac{c\mathfrak{c}|p|}{T_\infty}}}{|p|^{1+b+|\varsigma|}}|p_1|^{-\varrho} dp\\
        \le & c \int_0^{+\infty} e^{-\frac{c\mathfrak{c}|p'|}{T_\infty}} (c+|p'|^{1-\eta+\alpha}) |p'| \Big( \int_{\mathbb{R}} \frac{ |p_1|^{-\varrho} }{ (|p_1|^2+|p'|^2)^{\frac{1+b+|\varsigma|}{2}} } dp_1 \Big) d|p'| \\
        \le & c \int_0^{+\infty} e^{-\frac{c\mathfrak{c}|p'|}{T_\infty}} (c+|p'|^{1-\eta+\alpha}) |p'|^{1 -\varrho-b-|\varsigma|} \Big( \int_{\mathbb{R}} \frac{ |s|^{-\varrho} }{ (|s|^2 + 1 )^{\frac{1+b+|\varsigma|}{2}} } ds \Big) d|p'| \\
        \le & c \int |p'|^{1-\varrho-b-|\varsigma|} e^{-\frac{c\mathfrak{c}|p'|}{T_\infty}} d|p'| + c \int |p'|^{2-\varrho-b-|\varsigma|-\eta+\alpha} e^{-\frac{c\mathfrak{c}|p'|}{T_\infty}} d|p'|\\
        \le & c,
    \end{aligned}  
    \]
    since $0<\varrho<1$, $0\le\alpha\le\eta$ and $b+|\varsigma| \in (0,\frac{1}{2})$.

    {\it Case 2: $q\not=0$, $q_1 = 0$.}  We have that
    \[
        \int_{|p_1|<1} |k(p,q)|w_{-\varrho,\alpha}dp = \int_{\substack{|p_1|<1,\\|p'| \ge 2|q|}} |k(p,q)|w_{-\varrho,\alpha}dp + \int_{\substack{|p_1|<1,\\|p'|<2|q|}} |k(p,q)|w_{-\varrho,\alpha}dp =: I_1 +I_2.
    \]
    
    For $I_1$, $2|q| \le |p'|$, we have $\frac{|p|}{2} \le |p-q| \le \frac{3}{2}|p|$, then $|p|\sim|p-q|$. It holds from \eqref{eq:k:es} and the proving process in case 1 that
    \[
        I_1 \le c \int_{\substack{|p_1|<1,\\ |p'| \ge 2|q|}} \frac{(\mathfrak{c}+|p|)^{1-\eta+\alpha} e^{-\frac{c\mathfrak{c}|p|}{T_\infty}}}{|p|^{1+b+|\varsigma|}}|p_1|^{-\varrho} dp \le c,    
    \]
    since $0<\varrho<1$, $0\le\alpha\le\eta$ and $b+|\varsigma| \in (0,\frac{1}{2})$.

    For $I_2$, since $|p_1|\le 1$, $|p'|< 2|q|$ and $|q|=|q'|$, we have that $\mathfrak{c}+|p| \le c(\mathfrak{c}+|p'|) \le c(\mathfrak{c}+|q'|)$. Direct computation gives that
    \[
        (p \times q)^2 + \mathfrak{c}^2 |p-q|^2 = p_1^2(\mathfrak{c}^2+|q'|^2) + \mathfrak{c}^2 |p'-q'|^2 + \big( (p_2-q_2) q_3 - (p_3-q_3) q_2 \big)^2 =: p_1^2 m^2 + n^2, 
    \]
    where $m,n>0$, $m^2 := \mathfrak{c}^2+|q'|^2$, and 
    \[
    \begin{aligned}
        n^2 := & \mathfrak{c}^2 |p'-q'|^2 + \big( (p_2-q_2) q_3 - (p_3-q_3) q_2 \big)^2 \\
        = & \mathfrak{c}^2 |p'-q'|^2 + | (p'-q') \cdot l |^2 \\
        = & \mathfrak{c}^2 |p'-q'|^2 + | (p'-q') |^2 | q' |^2 \cos^2\theta, 
    \end{aligned}
    \]
    and $l = (q_3, -q_2)$, $\theta$ is the angle between vectors $(p'-q')$ and $l$.  

    Let $p_1 = \frac{n}{m}s$. By \eqref{eq:k:es}, it holds by a scaling transform that
    \[
    \begin{aligned}
        I_2 & \le c(\mathfrak{c}+|q'|)^{1-\eta+\alpha} \int_{\substack{|p_1|<1,\\ |p'| < 2|q|}} \frac{ e^{ - \frac{c\mathfrak{c}|p'-q'| }{ T_\infty } } |p_1|^{-\varrho} }{ (p_1^2m^2 + n^2)^{\frac{1}{2}} |p'-q'|^{b+|\varsigma|}} dp \\
        & \le c(\mathfrak{c}+|q'|)^{1-\eta+\alpha} \int_{|p'| < 2|q|} \frac{e^{-\frac{c\mathfrak{c}|p'-q'|}{T_\infty}}}{|p'-q'|^{b+|\varsigma|}} \Big( \int_{\mathbb{R}} \frac{ |p_1|^{-\varrho} }{ ( p_1^2 m^2 + n^2 )^{ \frac{1}{2} } } dp_1 \Big) dp_2 dp_3 \\
        & \le c(\mathfrak{c}+|q'|)^{1-\eta+\alpha} \int_{|p'| < 2|q|} \frac{e^{-\frac{c\mathfrak{c}|p'-q'|}{T_\infty}} m^{\varrho-1} n^{-\varrho} }{ |p'-q'|^{b+|\varsigma|}} \Big( \int_{0}^{+\infty} \frac{ |s|^{-\varrho} }{ ( s^2 + 1 )^{ \frac{1}{2} } } ds \Big) dp_2 dp_3 \\
        & \le c(\mathfrak{c}+|q'|)^{\varrho - \eta + \alpha} \int_{0}^{\infty} \int_{0}^{2\pi} \frac{e^{-\frac{c\mathfrak{c}r}{T_\infty}}}{r^{b+|\varsigma|-1+\varrho} (1 + |q'| |\cos\theta| )^{\varrho}} d\theta dr. 
    \end{aligned}
    \]
    If $|q'|$ is small, then $I_2$ is obviously bounded. For large $|q'|$, we have that
    \[
    \begin{aligned}
        I_2 & \le c(\mathfrak{c}+|q'|)^{\varrho - \eta + \alpha} \frac{1}{|q'|^\varrho}\int_{0}^{\infty} \frac{e^{-\frac{c\mathfrak{c}r}{T_\infty}}}{r^{b+|\varsigma|-1+\varrho}} 
        \int_{0}^{2\pi} \frac{1}{|\cos\theta|^{\varrho}} d\theta dr 
        \le c,
    \end{aligned}    
    \]
    since $0<\varrho<1$, $0\le\alpha\le\eta$ and $b+|\varsigma| \in (0,\frac{1}{2})$.
    
    {\it Case 3: $q_1 \ne 0$.} We have that
    \[
        \int_{|p_1|<1} |k(p,q)|w_{-\varrho,\alpha}dp 
        = \Big( \int_{\substack{|p_1|<1,\\|p'| \ge 2|q|}} + \int_{ \substack{ |p_1|<1,|p'|<2|q|,\\|q_1| \ge \max\{2, 2|q'|\} }  } + \int_{ \substack{ |p_1|<1,|p'|<2|q|,\\|q_1|<\max\{2, 2|q'|\} } } \Big) |k(p,q)|w_{-\varrho,\alpha}dp  =: I_3  + I_4 + I_5. 
    \]
    
     For $I_3$, $2|q| \le |p'|$, we have $\frac{|p|}{2} \le |p-q| \le \frac{3}{2}|p|$, then $|p|\sim|p-q|$. It holds from \eqref{eq:k:es} and the proving process in case 1 that
    \[
        I_3 \le c \int_{\substack{|p_1|<1,\\ |p'| \ge 2|q|}} \frac{(\mathfrak{c}+|p|)^{1-\eta+\alpha} e^{-\frac{c\mathfrak{c}|p|}{T_\infty}}}{|p|^{1+b+|\varsigma|}}|p_1|^{-\varrho} dp \le c,    
    \]
    since $0<\varrho<1$, $0\le\alpha\le\eta$ and $b+|\varsigma| \in (0,\frac{1}{2})$.
    
    For $I_4$, we have that $|q_1| \ge 2 > 2|p_1|$, $\frac{|q_1|}{2}\le|p-q|\le\frac{9}{2}|q_1|$, and $|q|\sim|q_1|$. It holds that
    \[
    \begin{aligned}
        I_4 & \le c e^{ - \frac{c\mathfrak{c}|q_1| }{ T_\infty } }\frac{(\mathfrak{c}+|q_1|)^{1-\eta+\alpha}}{|q_1|^{1+b+|\varsigma|}} \int_{0}^{2|q|} \Big( \int_{|p_1|<1} |p_1|^{-\varrho}dp_1\Big)  |p'| d|p'| \le c e^{ - \frac{c\mathfrak{c}|q_1| }{ T_\infty } } (\mathfrak{c}+|q_1|)^{1-\eta+\alpha} |q_1|^{1-b-|\varsigma|}
        \le c, 
    \end{aligned}    
    \]
    since $0<\varrho<1$ and $b+|\varsigma| \in (0,\frac{1}{2})$. 
    
    For $I_5$, it is easy to see that $\mathfrak{c}+|p|<\mathfrak{c}+|q| < c (\mathfrak{c}+|q'|)$ and $1<\frac{\mathfrak{c}^2+|q|^2}{\mathfrak{c}^2+|q'|^2}<c$. Then, it holds that 
    \[
    \begin{aligned}
        |p\times q|^2+\mathfrak{c}^2 |p-q|^2 
        & =: m^2(p_1+\tilde{a})^2+\tilde{n}^2,
    \end{aligned}    
    \]
    where $m,\tilde{n}>0$, and
    \[
    \begin{aligned}
        \tilde{a} & := -\frac{q_1}{\mathfrak{c}^2+|q'|^2}(\mathfrak{c}^2+p'\cdot q'),\\
        m^2 & := \mathfrak{c}^2+|q'|^2,\\
        \tilde{n}^2 & := \frac{\mathfrak{c}^2+|q|^2}{\mathfrak{c}^2+|q'|^2} \big((p_2q_3-p_3q_2)^2+\mathfrak{c}^2|p'-q'|^2\big) \sim n^2.
    \end{aligned}    
    \]
    Let $s := \frac{m}{n}p_1$, $\bar{a} := \frac{m}{n}\tilde{a}$. It holds by a scaling transform and Lemma \ref{lem:bar:a:bounded} that
    \[
    \begin{aligned}
        I_5 
        & \le c(\mathfrak{c}+|q'|)^{1-\eta+\alpha} \int_{|p'| < 2|q|} \frac{e^{-\frac{c\mathfrak{c}|p'-q'|}{T_\infty}}}{|p'-q'|^{b+|\varsigma|}} \Big( \int_{\mathbb{R}} \frac{ |p_1|^{-\varrho} }{ \big( m^2(p_1+\tilde{a})^2+n^2 \big)^{ \frac{1}{2} } } dp_1 \Big) dp_2 dp_3 \\
        & \le c(\mathfrak{c}+|q'|)^{1-\eta+\alpha} \int_{|p'| < 2|q|} \frac{e^{-\frac{c\mathfrak{c}|p'-q'|}{T_\infty}} m^{\varrho-1} n^{-\varrho} }{ |p'-q'|^{b+|\varsigma|}} \Big( \int_{-\infty}^{+\infty} \frac{ |s|^{-\varrho} }{ ( ( s+\bar{a})^2 + 1 )^{ \frac{1}{2} } } ds \Big) dp_2 dp_3\\
        & \le c(\mathfrak{c}+|q'|)^{\varrho - \eta + \alpha} \int_{0}^{\infty} \int_{0}^{2\pi} \frac{e^{-\frac{c\mathfrak{c}r}{T_\infty}}}{r^{b+|\varsigma|-1+\varrho} (1 + |q'| |\cos\theta| )^{\varrho}} d\theta dr.
    \end{aligned}    
    \]

    If $|q'|$ is small, then $I_5$ is obviously bounded. For large $|q'|$, we have that
    \[
    \begin{aligned}
        I_5 & \le c(\mathfrak{c}+|q'|)^{\varrho - \eta + \alpha} \frac{1}{|q'|^\varrho}\int_{0}^{\infty} \frac{e^{-\frac{c\mathfrak{c}r}{T_\infty}}}{r^{b+|\varsigma|-1+\varrho}} 
        \int_{0}^{2\pi} \frac{1}{|\cos\theta|^{\varrho}} d\theta dr 
        \le c,
    \end{aligned}    
    \]
    since $0<\varrho<1$, $0\le\alpha\le\eta$ and $b+|\varsigma| \in (0,\frac{1}{2})$. We finish the proof of lemma in all cases. 
\end{proof}

\begin{lemma}\label{lem:wk}
    The operators $w_{-\varrho,\alpha}K$ and $K w_{-\varrho,\alpha}$ are bounded from $L_p^2$ to $L_p^2$ for $0<\varrho < \frac{1}{2}$ and $0\le\alpha\le\frac{\eta}{2}$.
\end{lemma}
\begin{proof}
    We only prove the lemma for operator $w_{-\varrho,\alpha}K$, since the proof for $K w_{-\varrho,\alpha}$ is similar. 
    \[
        \begin{aligned}
            \|w_{-\varrho,\alpha}K f\|_{L_p^2}^2 = & \int w_{-2\varrho,2\alpha}(p) \Big(\int k(p,q)f(q)dq\Big)^2dp\\
            \le & \int w_{-2\varrho,2\alpha}(p) \Big(\int k(p,q)dq\Big)\Big(\int k(p,q)f^2(q)dq\Big)dp\\
            \le & \underset{p}{\sup}\int k(p,q)dq \cdot \int  f^2(q) \Big(\int k(p,q)w_{-2\varrho,2\alpha}(p)dp\Big)dq.
        \end{aligned}
    \]
    By (ii) of Proposition \ref{prop:k:weight} and Lemma \ref{lem:k:weight}, we have that
    \[
        \|w_{-\varrho,\alpha}K f\|_{L_p^2}^2 \le c \|f\|^2_{L_p^2},
    \]
    for any $0<\varrho<\frac{1}{2}$ and $0\le\alpha\le\frac{\eta}{2}$. 
\end{proof}
In order to obtain $L^\infty_{x,p,\beta}$ estimate of $h$, we present the following lemma which is based on the energy estimate \eqref{eq:L2:estimate} of the linear Boltzmann equation with damping terms obtained in the previous section.
\begin{lemma}\label{lem:weighted:h}
    Let $\eta$ be defined in \eqref{eq:eta} and $w_{-\varrho,\alpha}$ be the weight function in \eqref{eq:w}. Then for any $0<\varrho<\frac{1}{2}$ and $0\le\alpha\le\frac{\eta}{2}$, the solution of \eqref{eq:h:2}, where $\gamma > \tau > 0$ and both are sufficiently small, satisfies	
    \[
        \|w_{-\varrho,\alpha}\nu^{\frac{1}{2}} h\| \le c E_{\varrho,\alpha}^{\frac{1}{2}}, \qquad  \|w_{1,-1} \nu^{-\frac{1}{2}} h_x\| \le c E_0^{\frac{1}{2}}, 
    \]
    where 
    \[
    \begin{aligned}
        E_0 &:= \langle |\hat{p}_1| a_0,  a_0\rangle_+ +\|\zeta\|^2,\\
        E_{\varrho,\alpha} &:= \langle |\hat{p}_1| w_{-\varrho,\alpha}  a_0,  w_{-\varrho,\alpha} a_0\rangle_+ +\|\zeta\|^2+ \|\nu^{-\frac{1}{2}} w_{-\varrho,\alpha} \zeta \|^2.
    \end{aligned}
    \]
\end{lemma}
\begin{proof}
    Note that $E_0 \le c E_{\varrho,\alpha}$ for $\varrho,\alpha\ge 0$. From the energy estimate \eqref{eq:L2:estimate} for the solution of \eqref{eq:h:2}, 
    we have that
    \[
        \|h\|^2\le c(\|\zeta\|^2+\langle \hat{p}_1 a_0, a_0\rangle_+)= c E_0.
    \]

    Let $\theta(p)$ be a smooth cut-off function satisfying
    \[
        \theta(p)=0, \quad |p|>M; \qquad \qquad \theta(p)=1, \quad |p|<M-1,
    \]
    and is monotone with respect to $|p|$ when $M-1 \le |p| \le M$.
    We take the $L_{x,p}^2$ inner product of $\theta^2 w_{-\varrho,\alpha}^2 h$ and \eqref{eq:h:2}, then
    \[
        (\hat{p}_1 h_x, \theta^2 w_{-\varrho,\alpha}^2 h)-\tau(\hat{p}_1 h, \theta^2 w_{-\varrho,\alpha}^2 h)-( L h,\theta^2 w_{-\varrho,\alpha}^2 h)=( \zeta,\theta^2 w_{-\varrho,\alpha}^2 h)-\gamma( P_0^+\hat{p}_1 h, \theta^2 w_{-\varrho,\alpha}^2 h). 
    \]
    Since the first term
    \[
    \begin{aligned}
        (\hat{p}_1 h_x, \theta^2 w_{-\varrho,\alpha}^2 h)
        & = \frac{1}{2}\int_{\mathbb{R}^3}\theta^2 w_{-\varrho,\alpha}^2\hat{p}_1 h^2|_{x=0}^{x=\infty }dp\\
        & = - \frac{1}{2}\langle \theta^2 w_{-\varrho,\alpha}^2\hat{p}_1 h^0, h^0\rangle \\
        &=-\frac{1}{2}\langle \theta^2 w_{-\varrho,\alpha}^2\hat{p}_1 a_0, a_0\rangle_+ +\frac{1}{2}\langle \theta^2 w_{-\varrho,\alpha}^2 |\hat{p}_1| h^0, h^0\rangle_-,  
    \end{aligned}
    \]
    then we have that
    \[
    \begin{aligned}
        &\frac{1}{2}\langle \theta^2 w_{-\varrho,\alpha}^2 |\hat{p}_1| h^0, h^0\rangle_- -\tau(\hat{p}_1 h, \theta^2 w_{-\varrho,\alpha}^2 h) - ( L h,\theta^2 w_{-\varrho,\alpha}^2 h)\\
        =&( \zeta,\theta^2 w_{-\varrho,\alpha}^2 h)-\gamma( P_0^+\hat{p}_1 h, \theta^2 w_{-\varrho,\alpha}^2 h)+\frac{1}{2}\langle \theta^2 w_{-\varrho,\alpha}^2\hat{p}_1 a_0, a_0\rangle_+.
    \end{aligned}
    \]

    We note that
    \[
        L = - \nu +K,  \quad \bar{K}=K-\gamma P_0^+ \hat{p}_1.
    \]
    Since $\hat{p}_1$ is bounded, we have for sufficiently small $\tau>0$ that $| \tau(\hat{p}_1 h, \theta^2 w_{-\varrho,\alpha}^2 h) | \le \frac{1}{2} \| \nu^\frac{1}{2}\theta w_{-\varrho,\alpha}h \|^2$, and 
    \begin{equation}\label{eq:theta:es}
    \begin{aligned}
        & \| |\hat{p}_1|^\frac{1}{2} \theta w_{-\varrho,\alpha}h^0\|_-^2+\|\nu^\frac{1}{2}\theta w_{-\varrho,\alpha}h\|^2\\
        \le &c(\|\nu^{-\frac{1}{2}}\theta w_{-\varrho,\alpha} \bar{K}h\|\ \|\nu^\frac{1}{2}\theta w_{-\varrho,\alpha} h\|+\|\nu^{-\frac{1}{2}}\theta w_{-\varrho,\alpha} \zeta\|\ \|\nu^\frac{1}{2}\theta w_{-\varrho,\alpha} h\| +\| |\hat{p}_1|^\frac{1}{2} \theta w_{-\varrho,\alpha}a_0\|_+^2). 
    \end{aligned}
    \end{equation}

    By Lemma \ref{lem:wk}, the operators $w_{-\varrho,\alpha}\bar{K}$ and $\bar{K} w_{-\varrho,\alpha}$ are bounded from $L_p^2$ to $L_p^2$ for $0<\varrho < \frac{1}{2}$ and $0\le\alpha\le\frac{\eta}{2}$. 
    Letting $\theta (p ) \rightarrow 1$ in \eqref{eq:theta:es} and using the Cauchy's inequality, we have that
    \[
    \begin{aligned}
        &\||\hat{p}_1|^\frac{1}{2} w_{-\varrho,\alpha}h^0\|_-^2 + \|\nu^\frac{1}{2} w_{-\varrho,\alpha}h\|^2 \\
        \le & c ( \|\nu^{-\frac{1}{2}} w_{-\varrho,\alpha} \bar{K}h\|\ \|\nu^\frac{1}{2} w_{-\varrho,\alpha} h\|+\| \nu^{-\frac{1}{2}} w_{-\varrho,\alpha} \zeta\|\ \|\nu^\frac{1}{2} w_{-\varrho,\alpha} h\|+\||\hat{p}_1|^\frac{1}{2} w_{-\varrho,\alpha}a_0\|_+^2)\\
        \le & c ( \|h\| \| \nu^\frac{1}{2} w_{-\varrho,\alpha} h\|+\|\nu^{-\frac{1}{2}} w_{-\varrho,\alpha} \zeta\|\ \|\nu^\frac{1}{2} w_{-\varrho,\alpha} h\|+\||\hat{p}_1|^\frac{1}{2} w_{-\varrho,\alpha}a_0\|_+^2)\\
        \le & c \big( (E_0^{1/2} + E_{\varrho,\alpha}^{1/2}) \| \nu^\frac{1}{2}w_{-\varrho,\alpha} h \| + E_{\varrho,\alpha} \big) \\
        \le & \frac{1}{2} \|\nu^\frac{1}{2} w_{-\varrho,\alpha}h\|^2 + c E_{\varrho,\alpha}. 
    \end{aligned}
    \]
    From the above estimate, it holds that 
    \begin{equation}\label{eq:nu:h:esti}
        \| \nu^\frac{1}{2} w_{-\varrho,\alpha} h \|
        \le c E_{\varrho,\alpha}^{1/2}. 
    \end{equation}

    From \eqref{eq:nu:h}, we have that
    \begin{equation}\label{eq:nu:h:E}
        \| \nu^\frac{1}{2} h \| \le c E_0^{1/2}.
    \end{equation}
    By using \eqref{eq:h:2} and \eqref{eq:nu:h:E}, if $|p_1|<1$, 
    we have
    \begin{equation}\label{eq:nu:hx:le:esti}
    \begin{aligned}
        \|\nu^{-\frac{1}{2}} w_{1,-1}h_x\| \le c \|\nu^{-\frac{1}{2}} |\hat{p}_1|h_x\|&= c \|\nu^{-\frac{1}{2}}(\tau  \hat{p}_1 h+Lh+\zeta-\gamma P_0^+ \hat{p}_1h)\|\\
        & \le c \big( \|\nu^{-\frac{1}{2}}\tau  \hat{p}_1h\|+\|\nu^{-\frac{1}{2}}\nu h\|+\|\nu^{-\frac{1}{2}}\zeta\|+\|\nu^{-\frac{1}{2}}\bar{K}h\| \big) \\
        &\le c (\|h\|+\|\nu^{\frac{1}{2}} h\|+\|\zeta\|)\\
        &\le c E_0^{1/2}.
    \end{aligned}
    \end{equation}
    In addition, if $|p_1| \ge 1$, we have
    \begin{equation}\label{eq:nu:hx:ge:esti}
    \begin{aligned}
        \|\nu^{-\frac{1}{2}} w_{1,-1}h_x\|=\|\nu^{-\frac{1}{2}} p_0^{-1}h_x\|&=\|\nu^{-\frac{1}{2}}(p_0^{-1} \tau  h+\frac{1}{\mathfrak{c} p_1}Lh+\frac{1}{\mathfrak{c} p_1}\zeta-\gamma \frac{1}{\mathfrak{c} p_1} P_0^+\hat{p}_1 h)\|\\
        &\le c \big( \|\nu^{-\frac{1}{2}}\tau h\|+\|\nu^{-\frac{1}{2}}\nu h\|+\|\nu^{-\frac{1}{2}}\zeta\|+\|\nu^{-\frac{1}{2}}\bar{K}h\| \big) \\
        &\le c (\|h\|+\|\nu^{\frac{1}{2}} h\|+\|\zeta\|)\\
        &\le c E_0^{1/2}.
    \end{aligned}
    \end{equation}
    By \eqref{eq:nu:h:esti}, \eqref{eq:nu:hx:le:esti} and \eqref{eq:nu:hx:ge:esti}, the lemma is proved. 
\end{proof}

Now we are ready to obtain the main estimate in this section on $\|h\|_{L_{x,p,\beta}^\infty}$. 

\begin{theorem}\label{thm:h:infty}
    Let $\eta$ be defined in \eqref{eq:eta}. Suppose $\eta \in [ \frac{2}{3}, 1 ]$, $0 < \varrho < \frac{1}{2}$, $ 1-\eta \le \alpha \le \frac{\eta}{2}$ and $\beta>\frac{3+2\alpha}{2}$. Then the solution of \eqref{eq:h:2}, where $\gamma > \tau > 0$ and both are sufficiently small, satisfies 
    \begin{equation}\label{lem:h:beta}
        \|h\|_{L_{x,p,\beta}^\infty }\le c(\|\nu^{-1}\zeta\|_{L_{x,p,\beta}^\infty}+\|\zeta\|+\|\nu^{-\frac{1}{2}}w_{-\varrho,\alpha}\zeta\|+\|a_0\|_{L_{p,\beta}^\infty(\mathbb{R}^3_+)}).
    \end{equation}
\end{theorem}
\begin{proof}
    Using \eqref{def:sol:h}, properties of $U$ in Lemma \ref{lem:U}, and properties of $\nu$ and $K$, as well as $\bar{K}$, in Proposition \ref{prop:L}, we have that 
    \begin{equation}\label{eq:h:beta}
    \begin{aligned}
        \|h\|_{L_{x,p,\beta}^\infty}&=\underset{x\in \mathbb{R}^+,p\in \mathbb{R}^3}{\sup}|p_{0}^{\beta}(\tilde{a}+U(\bar{K}h+\zeta))|\\
        &\le \underset{x\in \mathbb{R}^+,p\in \mathbb{R}^3}{\sup}|p_{0}^{\beta}\tilde{a}|+\underset{x\in \mathbb{R}^+,p\in \mathbb{R}^3}{\sup}|p_{0}^{\beta}U(\bar{K}h+\zeta)|\\
        &\le \underset{p\in \mathbb{R}^3_+}{\sup}|p_{0}^{\beta}a_0(p)|+\underset{x\in \mathbb{R}^+,p\in \mathbb{R}^3}{\sup}|p_{0}^{\beta}U(\bar{K}h+\zeta)|\\
        &=\|a_0\|_{L_{p,\beta}^\infty(\mathbb{R}^3_+)}+\|U(\bar{K}h+\zeta)\|_{L^\infty_{x,p,\beta}}\\
        &\le \|a_0\|_{L_{p,\beta}^\infty(\mathbb{R}^3_+)}+c(\|\nu^{-1}\bar{K}h\|_{L^\infty_{x,p,\beta}}+\|\nu^{-1}\zeta\|_{L^\infty_{x,p,\beta}}) \\
        &\le \|a_0\|_{L_{p,\beta}^\infty(\mathbb{R}^3_+)}+c(\|h\|_{L^\infty_{x,p,\beta-\eta-\frac{a}{2}}}+\|\nu^{-1}\zeta\|_{L^\infty_{x,p,\beta}}). 
    \end{aligned}
    \end{equation}

    Iterating \eqref{eq:h:beta} gives
    \begin{equation}\label{eq:h:beta:estimate}
    \begin{aligned}
        \|h\|_{L^\infty_{x,p,\beta}}&\le c(\|a_0\|_{L_{p,\beta}^\infty(\mathbb{R}^3_+)}+\|h\|_{L_{x,p,c'-\eta-\frac{a}{2}}^\infty }+ \vertiii{\zeta} ) \\
        &\le c(\|a_0\|_{L_{p,\beta}^\infty(\mathbb{R}^3_+)}+\|h\|_{L^\infty_{x,p}} + \vertiii{\zeta} ),
    \end{aligned}
    \end{equation}
    where $0<c'<\eta+\frac{a}{2}$ and 
    \begin{equation}\label{eq:vertiii:zeta}
        \vertiii{\zeta} := \|\nu^{-1}\zeta\|_{L^\infty_{x,p,\beta}}+\|\nu^{-1}\zeta\|_{L_{x,p,\beta-\eta-\frac{a}{2}}^\infty} + \cdots + \|\nu^{-1}\zeta\|_{L_{x,p,c'}^\infty }\le c\|\nu^{-1}\zeta\|_{L_{x,p,\beta}^\infty }.    
    \end{equation}
    
    It remains to estimate $\|h\|_{L^\infty_{x,p}}$. By using the expression of $h$, Lemma \ref{lem:U} and noticing that the operator $\bar{K}$ is bounded from $L^2_p(L^\infty_x)$ to $L^\infty_{x,p}$, we obtain that
    \begin{equation}\label{eq:infty:infty}
    \begin{aligned}
        \|h\|_{L^\infty_{x,p}}&=\|\tilde{a}+U(\bar{K}h+\zeta)\|_{L^\infty_{x,p}}\\
        &\le\|\tilde{a}\|_{L^\infty_{x,p}}+\|U(\bar{K}h+\zeta)\|_{L^\infty_{x,p}}\\
        &\le\| a_0 \|_{L^\infty_{p,+}}+c\|\nu^{-1}\bar{K}h\|_{L^\infty_{x,p}} + c \|\nu^{-1}\zeta\|_{L^\infty_{x,p}} \\
        &\le\| a_0 \|_{L^\infty_{p,+}}+c(\|h\|_{L^2_{p}(L^\infty_x)}+\|\nu^{-1}\zeta\|_{L^\infty_{x,p}}).
    \end{aligned}
    \end{equation}
    
    By Lemma \ref{lem:wk}, the operator $\bar{K} w_{-\varrho',\alpha'}$ are bounded from $L_p^2$ to $L_p^2$ for $0 < \varrho' < \frac{1}{2}$ and $0 \le \alpha' \le \frac{\eta}{2}$. By \eqref{def:sol:h} and Lemma \ref{lem:U}, we have for $\eta\in [ \frac{2}{3}, 1 ] $, $\frac{1}{4}<\varrho'<\frac{1}{2}$, $\frac{1}{2}-\frac{\eta}{4}\le\alpha'\le\frac{\eta}{2}$ that 
    \begin{equation}\label{eq:h:2infty}
    \begin{aligned}
        \|h\|_{L^2_p(L^\infty_x)}&=\|\tilde{a}+U(\bar{K}h+\zeta)\|_{L^2_p(L^\infty_x)}\\
        &\le\|\tilde{a}\|_{L^2_p(L^\infty_x)}+\|U(\bar{K}h+\zeta)\|_{L^2_p(L^\infty_x)}\\
        &\le \| a_0 \|_{L^2_{p} ( \mathbb{R}^3_+ ) } + c \| \nu^{-1} \bar{K} w_{-\varrho',\alpha'}w_{\varrho',-\alpha'}h\|_{L^2_p(L^\infty_x)} + c \|\nu^{-1}\zeta\|_{L^2_p(L^\infty_x)} \\
        &\le\|a_0 \|_{L^2_{p} ( \mathbb{R}^3_+ ) } + c ( \| w_{\varrho',-\alpha'} h \|_{ L^2_p(L^\infty_x)}+\|\nu^{-1}\zeta\|_{L^2_p(L^\infty_x)}) .
    \end{aligned}
    \end{equation}
    
    It holds for $\beta >\frac{3}{2}$ that 
    \begin{equation}\label{eq:a0}
    \begin{aligned}
        \| a_0 \|_{L^2_{p} (\mathbb{R}^3_+)}^2 & = \int_{p_1>0}a_0^2(p)dp\\
        & \le \int_{p_1>0}p_0^{-2\beta} p_0^{2\beta} a_0^2(p)dp\\
        & \le \underset{p_1>0}{\sup} |p_0^{2\beta} a_0^2 | \cdot \int_{p_1>0}p_0^{-2\beta}dp\\
        & \le c \| a_0 \|^2_{ L^\infty_{p,\beta} (\mathbb{R}^3_+) },
    \end{aligned}
    \end{equation}
    and
    \begin{equation}\label{eq:nuh}
    \begin{aligned}
        \| \nu^{-1}\zeta \|_{L^2_p(L^\infty_x)} = & \big(\int_{\mathbb{R}^3}(p_0^{\beta}\underset{x\in \mathbb{R}^+}{\sup}|\nu^{-1}\zeta|)^2p_0^{-2\beta}dp\big)^{\frac{1}{2}}\\
        \le & \underset{x\in \mathbb{R}^+,p\in \mathbb{R}^3}{\sup}|p_0^{\beta}\nu^{-1}\zeta| \cdot \big(\int p_0^{-2\beta}dp\big)^{\frac{1}{2}}\\
        \le & c\|\nu^{-1}\zeta\|_{L_{x,p,\beta}^\infty }. 
    \end{aligned}
    \end{equation}

    Note that $\eta\in [ \frac{2}{3}, 1 ]$,  $\frac{1}{4}<\varrho'<\frac{1}{2}$ and $\frac{1}{2}-\frac{\eta}{4}\le\alpha'\le\frac{\eta
    }{2}$, i.e. $0<1-2\varrho'<\frac{1}{2}$ and $0\le1-\eta\le 1-2\alpha'\le\frac{\eta}{2}$. By Lemma \ref{lem:estimate:f}, Lemma \ref{lem:weighted:h} and H\"{o}lder's inequality, it holds that
    \begin{equation}\label{eq:two:infty}
    \begin{aligned}
        \|w_{\varrho',-\alpha'}h\|_{L^2_p(L^\infty_x)}^2 & = \int_{\mathbb{R}^3} w_{\varrho',-\alpha'}^2 \|h\|_{L^\infty_x}^2 d p \\
        & \le c\int_{\mathbb{R}^3} \big( w_{\varrho',-\alpha'}^2w_{-1,1}\nu^\frac{1}{2}\|h(x,p)\|_{L^2_x}\big) \big( w_{1,-1}\nu^{-\frac{1}{2}}\|h_x(x,p)\|_{L^2_x}   \big) dp \\
        &\le c\Big( \int_{\mathbb{R}^3}\| w_{-(1-2\varrho'),1-2\alpha'}\nu^\frac{1}{2}h(x,p)\|_{L^2_x}^2 dp\Big) ^\frac{1}{2}\Big( \int_{\mathbb{R}^3}\| w_{1,-1}\nu^{-\frac{1}{2}}h_x(x,p)\|_{L^2_x}^2 dp\Big) ^\frac{1}{2} \\
        & = c \| w_{-(1-2\varrho'),1-2\alpha'}\nu^\frac{1}{2}h(x,p)\|_{L^2_{x,p}} \  \| w_{1,-1}\nu^{-\frac{1}{2}}h_x(x,p)\|_{L^2_{x,p}}\\
        &\le c E_0^{\frac{1}{2}} E_{1-2\varrho',1-2\alpha'}^{\frac{1}{2}} = c E_0^{\frac{1}{2}} E_{\varrho,\alpha}^{\frac{1}{2}},
    \end{aligned}
    \end{equation}
    by letting $\varrho := 1-2\varrho'$, $\alpha := 1-2\alpha'$. Obviously, $0 < \varrho < \frac{1}{2}$ and $0 \le 1-\eta \le \alpha \le \frac{\eta}{2}$. 

    By using \eqref{eq:h:beta:estimate} - \eqref{eq:two:infty} and the definition of $E_0$ and $E_{\varrho,\alpha}$ in Lemma \ref{lem:weighted:h}, when $0<\varrho <\frac{1}{2}$ and $0 \le 1-\eta \le \alpha \le \frac{\eta}{2}$, we have that
    \[
    \begin{aligned}
        \|h\|_{L_{x,p,\beta}^\infty}&\le c(\| a_0(p) \|_{ L^\infty_{p,\beta} (\mathbb{R}^3_+)}+E_0^{\frac{1}{4}} E_{\varrho,\alpha}^{\frac{1}{4}} + \|\nu^{-1}\zeta\|_{L_{x,p,\beta}^\infty })\\
        &\le c(\| a_0(p) \|_{ L^\infty_{p,\beta} (\mathbb{R}^3_+)}+\||\hat{p}_1|^\frac{1}{2}w_{-\varrho,\alpha}a_0\|_+ +\|\zeta\|+\|\nu^{-\frac{1}{2}}w_{-\varrho,\alpha}\zeta\|+\|\nu^{-1}\zeta\|_{L_{x,p,\beta}^\infty}).
    \end{aligned}
    \]

    Since $\hat{p}_1$ is bounded, we have for $0<\varrho <\frac{1}{2}$, $\beta >\frac{3+2\alpha}{2}$ that
    \begin{equation}\label{eq:wp0a0}
    \begin{aligned}
        \||\hat{p}_1|^\frac{1}{2}w_{-\varrho,\alpha}a_0\|_+ \le & \big(\int_{p_1>0}w_{-2\varrho,2\alpha}p_0^{-2\beta}p_0^{2\beta}a_0^2dp\big)^{\frac{1}{2}}\\
        \le & \underset{p\in \mathbb{R}^3,p_1>0}{\sup} | p_0^{\beta}a_0 | \cdot \big(\int w_{-2\varrho,2\alpha}p_0^{-2\beta}dp\big)^{\frac{1}{2}} \\
        \le & c\| a_0 \|_{ L^\infty_{p,\beta} (\mathbb{R}^3_+)}.
    \end{aligned}
    \end{equation}

    Therefore, by \eqref{eq:wp0a0}, when $0<\varrho <\frac{1}{2}$, $0 \le 1-\eta \le \alpha \le \frac{\eta}{2}$ and $\beta >\frac{3+2\alpha}{2}$,  we have that
    \[
    \begin{aligned}
        \|h\|_{L_{x,p,\beta}^\infty}&
        &\le c(\| a_0(p) \|_{ L^\infty_{p,\beta} (\mathbb{R}^3_+)} +\|\zeta\|+\|\nu^{-\frac{1}{2}}w_{-\varrho,\alpha}\zeta\|+\|\nu^{-1}\zeta\|_{L_{x,p,\beta}^\infty}).
    \end{aligned}
    \]
    This completes the proof of the theorem.
\end{proof}

\section{Existence of the nonlinear problem with damping}\label{sec:enp}

In this section, we will prove the existence of solution to the nonlinear Boltzmann equation \eqref{eq:h:2} and \eqref{eq:zeta} with damping. 
\begin{theorem}\label{thm:non:exist}
    Suppose that the scattering kernel $\sigma$ satisfies \eqref{eq:sigma}, \eqref{eq:ab} and \eqref{eq:gamma}, $\eta$ defined in \eqref{eq:eta} belongs to $[\frac{2}{3}, 1]$, $\beta > 2$. There exist positive constants $\tau$, $\gamma$, $\epsilon_0$, $c_0$, such that for any $a_0$ satisfying
    \[
        \| a_0(p) \|_{ L^\infty_{p,\beta} (\mathbb{R}^3_+)}<\epsilon_0,
    \]
    there exists a unique solution $h$ to the nonlinear problem \eqref{eq:h:2}, \eqref{eq:zeta}, which satisfies
    \[
        \|h\|_{L^\infty_{x,p,\beta}}<c_0. 
    \]
\end{theorem}
\begin{proof}
    Using the property of $\Gamma(h_1, h_2)$ in 
    Proposition \ref{prop:gamma}, it holds for any $\beta \ge 0$ that
    \begin{equation}\label{eq:nu:zeta}
        \|\nu^{-1}\zeta\|_{L_{x,p,\beta}^\infty } = \|e^{-\tau x}\nu^{-1} \Gamma(h)\|_{L_{x,p,\beta}^\infty }\le c\|h\|_{L_{x,p,\beta}^\infty }^2,
    \end{equation}
    and for any $\beta>\frac{3+a}{2}$, 
    \begin{equation}\label{eq:zeta:es}
    \begin{aligned}
        \|\zeta\|=\|e^{-\tau x} \Gamma(h)\|&=(\int_0^\infty \int_{\mathbb{R}^3} e^{-2\tau x}\Gamma^2(h)dpdx)^\frac{1}{2}\\
        &\le \|\nu^{-1}\Gamma(h) \|_{L_{x,p,\beta}^\infty } (\int_0^\infty \int_{\mathbb{R}^3} e^{-2\tau x} p_0^{-2\beta}\nu^{2}dpdx)^\frac{1}{2}\\
        &\le c\|h\|_{L_{x,p,\beta}^\infty }^2.
    \end{aligned}
    \end{equation}

    We will choose $\alpha \in [1-\eta,\eta/2]$ later. For any $0<\varrho<\frac{1}{2}$, $\beta>\frac{3+2\alpha+\frac{a}{2}}{2}$, we have that
    \begin{equation}\label{eq:zeta:nu:weight:es}
    \begin{aligned}
        \| \nu^{-\frac{1}{2}} w_{-\varrho,\alpha} \zeta \|^2 = \|\nu^{-\frac{1}{2}}w_{-\varrho,\alpha}e^{-\tau x}\Gamma(h)\|^2 & = \int_0^\infty \int_{\mathbb{R}^3} \nu^{-1}w_{-2\varrho,2\alpha}e^{-2\tau x}\Gamma^2(h)dpdx \\
        & \le \|\nu^{-1}\Gamma(h) \|_{L_{x,p,\beta}^\infty }^2 \cdot \int_0^\infty \int_{\mathbb{R}^3} \nu e^{-2\tau x}w_{-2\varrho,2\alpha}p_0^{-2\beta} dpdx \\
        &\le c\|h\|_{L_{x,p,\beta}^\infty }^4. 
    \end{aligned}
    \end{equation}

    By \eqref{eq:ab} and \eqref{eq:eta}, it is easy to see that $a/4 \le 1-\eta\le \alpha$. Then
    \[
        \max \{\frac{3+a}{2}, \frac{3+2\alpha+\frac{a}{2}}{2}\} = \frac{3+2\alpha+\frac{a}{2}}{2}. 
    \] 
    By \eqref{eq:ab}, \eqref{eq:gamma} and \eqref{eq:eta}, it holds when $\eta \in [\frac{2}{3}, 1]$ that $0\le 1-\eta \le \frac{1}{3}$ and $0 \le a \le \frac{2}{3}$. 
    
    For any $\beta>2$, there exists a suitable small $\varepsilon > 0$ such that $\beta > 2+\varepsilon$. 
    Chose $\alpha$ as close to $1-\eta$ as possible, such that $1-\eta \le \alpha < 1-\eta +\varepsilon$, then 
    \[
        \frac{3+2\alpha+\frac{a}{2}}{2} < \frac{3 + 2(1 - \eta + \varepsilon) + \frac{a}{2}}{2} \le 2 + \varepsilon < \beta. 
    \]
    Now all above restrictions on $\beta$ hold. 

    Therefore, by choosing a fixed $0<\varrho<\frac{1}{2}$ and $\alpha$ close to $1-\eta$, and substituting \eqref{eq:nu:zeta}, \eqref{eq:zeta:es}, \eqref{eq:zeta:nu:weight:es} into \eqref{lem:h:beta}, we obtain that
    \begin{equation}\label{eq:h:beta:two}
        \|h\|_{L^\infty_{x,p,\beta}}
        \le c(\|h\|_{L_{x,p,\beta}^\infty }^2+\| a_0(p) \|_{ L^\infty_{p,\beta} (\mathbb{R}^3_+)}).
    \end{equation}
    When $\| a_0(p) \|_{ L^\infty_{p,\beta} (\mathbb{R}^3_+)}$ is sufficiently small, the contraction mapping theorem and \eqref{eq:h:beta:two} yield the existence of the unique solution to \eqref{eq:h:2}, \eqref{eq:zeta}.
\end{proof}

\begin{corollary}
    Suppose that the scattering kernel $\sigma$ satisfies \eqref{eq:sigma}, \eqref{eq:ab} and \eqref{eq:gamma}, $\eta$ defined in \eqref{eq:eta} belongs to $[\frac{2}{3}, 1]$, $\beta > 2$. There exist positive constants $\tau$, $\gamma$, $\epsilon_0$, $c_0$, such that for any $F_0$ satisfying
    \[
        |F_0(p)-J_\infty(p)|\le \epsilon_0 W_{\beta}(p),
    \]
    there exists a unique solution $F$ to problem \eqref{eq:F:damping}, which satisfies
    \[
        |F(x,p)-J_\infty(p)|\le c_0 e^{-\tau x} W_{\beta}(p).
    \]
\end{corollary}
\begin{remark}
    If $\mathcal{M}^\infty < -1$, then $n^+=0$, and $P_0^+ \equiv 0$. The corollary presents the portion of Theorem \ref{thm:main:result} for the case where $\mathcal{M}^\infty < -1$. 
\end{remark}

\section{Existence of the problem without damping}\label{sec:without:damping}

In this section, we will discuss the existence of solutions to the problem without damping under suitable solvability conditions. After establishing the existence and energy estimates of the solutions to the problem with damping in the $L^2 \cap L^\infty$ space, the proof of existence for the solutions to the undamped problem proceeds by following the same strategy as in \cite{Ukai:Yang:Yu}. 

By Theorem \ref{thm:lin:existence}, define the linear operator $\mathbb{I}^\gamma: L_{\hat{p}_1,+}^2 \rightarrow L_{|\hat{p}_1|}^2$ for sufficiently small $\gamma>0$ as follows:
\[
    \mathbb{I}^\gamma(a_0) := f(0,\cdot), 
\]
where $f(x,p)$ is the solution of
\begin{equation}\label{eq:lin:ope}
    \left\{ \begin{array}{ll}
    	\hat{p}_1 f_x=Lf-\gamma P^+_0 \hat{p}_1 f,\quad & x\in \mathbb{R}^+,p\in\mathbb{R}^3, \\
    	f(0,p)=a_0(p),& p_1\ge0,\\
        \displaystyle \lim_{x\to +\infty} f = 0,  & p\in \mathbb{R}^3,
    \end{array} \right.
\end{equation}
and $f(0,\cdot)$ represents the mapping $p\mapsto f(0,p)$ for $p\in\mathbb{R}^3$. 

Similarly, we define the nonlinear operator $\mathbb{J}^\gamma: L_{\hat{p}_1,+}^2 \cap L_{p,\beta}^\infty(\mathbb{R}^3_+) \rightarrow L_{|\hat{p}_1|}^2 $ through: 
\[
    \mathbb{J}^\gamma(a_0) := f(0,\cdot), 
\]
where $f(x,p)$ is the solution of 
\begin{equation}\label{eq:non:ope}
    \left\{ \begin{array}{ll}
    	\hat{p}_1 f_x=Lf-\gamma P^+_0 \hat{p}_1 f +\Gamma(f,f),\quad & x\in \mathbb{R}^+,p\in\mathbb{R}^3\\
    	f(0,p)=a_0(p),& p_1\ge0,\\
        \displaystyle \lim_{x\to +\infty} f = 0,  & p\in \mathbb{R}^3.
    \end{array} \right.
\end{equation}

Employing the energy estimates in Theorem \ref{thm:lin:existence} and \eqref{eq:h:beta:two} in Theorem \ref{thm:non:exist} to the linear and nonlinear relativistic Boltzmann equation with damping, we can establish the following lemma concerning $\mathbb{I}^\gamma$ and $\mathbb{J}^\gamma$.
\begin{lemma}
    Suppose that $\gamma$ is sufficiently small and $\beta>2$. Then the solution operator $\mathbb{I}^\gamma$ and $\mathbb{J}^\gamma$ satisfy: 
    
    {\rm (i)} $\mathbb{I}^\gamma :a_0 \in L_{\hat{p}_1,+}^2 \mapsto \mathbb{I}^\gamma(a_0) \in L_{|\hat{p}_1|}^2$ is bounded. 

    {\rm (ii)} $\mathbb{J}^\gamma:a_0 \in L_{\hat{p}_1,+}^2 \cap L_{p,\beta}^\infty(\mathbb{R}^3_+) \mapsto \mathbb{J}^\gamma(a_0) \in L_{|\hat{p}_1|}^2 $ is bounded, if $a_0$ is sufficiently small. 
\end{lemma}

The following theorem ensures that the solution to the relativistic Boltzmann equation with damping term is precisely the same as the one without damping term, if an implicit solvability condition \eqref{eq:sol:con:non} on the boundary value holds. 
\begin{theorem}\label{thm:persistance:init:cond}
    Suppose that 
    \begin{equation}\label{eq:sol:con:non}
        P_0^+ \hat{p}_1 \mathbb{J}^\gamma(a_0) \equiv 0    
    \end{equation}
    for some $\gamma>0$. Then any solution of \eqref{eq:f} is a solution of \eqref{eq:linearized:F}. 
\end{theorem}
\begin{proof}
    We project \eqref{eq:non:ope} to its macroscopic component. Since $L$ and $\Gamma$ are microscopic quantities, we have that
    \[
        \partial_x P_0^+ P_0 \hat{p}_1 f = -\gamma P_0^+ P_0 P_0^+\hat{p}_1 f.
    \]
    Since projection operators $P_0$ and $P_0^+$ satisfy
    \[
        P_0^+ P_0 P_0^+ = P_0^+ P_0 = P_0^+,
    \]
    then we have a linear differential equation for $P_0^+ \hat{p}_1 f$,
    \[
        \partial_x P_0^+ \hat{p}_1 f = -\gamma P_0^+ \hat{p}_1 f.  
    \]
    Therefore,
    \[
        P_0^+ \hat{p}_1 f(x,p) = P_0^+ \hat{p}_1 f(0,p) e^{-\gamma x}.
    \]

    If the boundary condition satisfies $P_0^+ \hat{p}_1 f|_{x=0}=0$, then we have that
    \[
        P_0^+ \hat{p}_1 f(x,p) \equiv 0,\quad x\ge0.
    \]
    That is to say, the solution of \eqref{eq:non:ope} is a solution of \eqref{eq:linearized:F} under the condition \eqref{eq:sol:con:non}.
\end{proof}

Similarly, we have the following theorem on the linearized equation.
\begin{theorem}
    Suppose that
    \begin{equation}\label{eq:sol:con:lin}
        P_0^+ \hat{p}_1 \mathbb{I}^\gamma(a_0) \equiv 0    
    \end{equation}
    for some $\gamma>0$. Then any solution of \eqref{eq:lin:ope} is a solution of 
    \begin{equation}\label{eq:lin:nodamping}
    \left\{ 
        \begin{array}{ll}
            \hat{p}_1 f_x=Lf,\quad & x\in \mathbb{R}^+,p\in\mathbb{R}^3, \\
            f(0,p)=a_0(p),& p_1\ge0, \\
            \displaystyle \lim_{x\to +\infty} f(x,p) = 0,  & p\in \mathbb{R}^3. 
        \end{array} 
        \right.
    \end{equation}
\end{theorem}

Next, by analyzing the solvability conditions \eqref{eq:sol:con:lin} for the linear equation and \eqref{eq:sol:con:non} for the nonlinear equation separately, we will classify the boundary conditions. 

\subsection{Classification of $P_0^+ \hat{p}_1 \mathbb{I}^\gamma(a_0) = 0$}

Since $\mathbb{I}^\gamma$ is a bounded linear operator, the function $P_0^+ \hat{p}_1 \mathbb{I}^\gamma(a_0)$ is a bounded linear map from $L_{\hat{p}_1,+}^2$ to a finite dimensional space. According to the Mach number of the far field Maxwellian, we have the following theorem on the co-dimensions of the boundary value satisfying \eqref{eq:sol:con:lin}. 

\begin{theorem}\label{lem:non:codi}
    The classification of the boundary value which satisfies the solvability condition \eqref{eq:sol:con:lin} can be summarized in Table \ref{tab_lin}.
    \begin{table}[H]
    \renewcommand\arraystretch{1.25}
    \setlength{\abovecaptionskip}{-0.1cm}
    \setlength{\belowcaptionskip}{0.2cm}
        \centering
        \caption{Classification of the solvability condition for the linear problem}
        \label{tab_lin}
        \begin{tabular}{ll}
        \toprule
             $\mathcal{M}^\infty<-1$ \quad & $\codim (\{a_0 \in L_{\hat{p}_1,+}^2 \mid P_0^+ \hat{p}_1 \mathbb{I}^\gamma(a_0) = 0\})=0$   \\
             $-1<\mathcal{M}^\infty<0$  & $\codim (\{a_0 \in L_{\hat{p}_1,+}^2 \mid P_0^+ \hat{p}_1 \mathbb{I}^\gamma(a_0) = 0\})=1$   \\
             $0<\mathcal{M}^\infty<1$  & $\codim (\{a_0 \in L_{\hat{p}_1,+}^2 \mid P_0^+ \hat{p}_1 \mathbb{I}^\gamma(a_0) = 0\})=4$   \\
             $\mathcal{M}^\infty>1$  & $\codim (\{a_0 \in L_{\hat{p}_1,+}^2 \mid P_0^+ \hat{p}_1 \mathbb{I}^\gamma(a_0) = 0\})=5$\\   
        \bottomrule
        \end{tabular}
    \end{table}
\end{theorem}
\begin{proof}
    Let $\Xi_i$, $i=1,\cdots,5$, be the eigenvectors of the operator $B$, which is defined in \eqref{eq:B} or \eqref{eq:B:2}, on $\mathcal{N}$. Then by \eqref{eq:lambda}, 
    \[
        B \Xi_i = \lambda_i \Xi_i,  
    \]
    where 
    \[
        \lambda_1 = \frac{a_3u_0u_{\infty,1}-\mathfrak{c}^2M}{\mathfrak{c} A_2},\qquad \lambda_2 =\lambda_3=\lambda_4= \mathfrak{c}\frac{u_{\infty,1}}{u_0} ,\qquad 
        \lambda_5 = \frac{a_3u_0u_{\infty,1}+\mathfrak{c}^2M}{\mathfrak{c} A_2},
    \]
    and $M,a_3,A_2$ are given in \eqref{eq:M}, \eqref{eq:a3}, \eqref{eq:A2}, respectively.
    We aim to demonstrate that the dimension of the non-trivial solution $a_0$ to 
    \begin{equation}\label{eq:no:zero}
        P_0^+ \hat{p}_1 \mathbb{I}^\gamma (a_0) \ne 0    
    \end{equation}
    is exactly $n^+$. Since $P_0^+ \hat{p}_1 \mathbb{I}^\gamma$ is a bounded linear operator from $L_{\hat{p}_1,+}^2$ to $\mathbb{R}^{n^+}$, the dimension of non-trivial solution is at most $n^+$. Thus finding $n^+$ independent non-trivial solutions for \eqref{eq:no:zero} is sufficient.

    For this, we introduce auxiliary functions
    \[
        j_i^\gamma := e^{-\gamma x} \Xi_i(p), \quad i=1,\cdots,5.
    \]

    When $\mathcal{M}^\infty<-1$, the matrix $P_0^+$ is a zero matrix, i.e. $\dim (P_0^+)=0$. Therefore, condition $\eqref{eq:sol:con:lin}$ holds for any $a_0 \in L_{\hat{p}_1,+}^2$. Therefore, no additional conditions need to be imposed on $a_0$. That is to say
    \[
        \codim ( \{ a_0 \in L_{\hat{p}_1,+}^2 \mid P_0^+ \hat{p}_1 \mathbb{I}^\gamma(a_0) = 0\})=0.
    \]

    When $-1<\mathcal{M}^\infty<0$, the range of $P_0^+$ is spanned by $\Xi_5$. Therefore, $\dim(P_0^+)=n^+=1$.
    Direct calculation reveals that function $j_5^\gamma$ satisfies
    \begin{equation}
    \begin{aligned}
        \hat{p}_1 \partial_x j_5^\gamma - Lj_5^\gamma + \gamma P_0^+ \hat{p}_1 j_5^\gamma & = -\gamma \hat{p}_1 e^{-\gamma x} \Xi_5(p) + \gamma e^{-\gamma x} P_0^+ \hat{p}_1 \Xi_5(p)\\
        & = -\gamma e^{-\gamma x} (P_0 + P_1) \hat{p}_1  \Xi_5(p) + \gamma e^{-\gamma x} P_0^+ \hat{p}_1 \Xi_5(p)\\
        & = -\gamma e^{-\gamma x} P_1 \hat{p}_1 \Xi_5(p).
    \end{aligned}
    \end{equation}

    Let $J_5^\gamma (x,p) := j_5^\gamma (x,p) + k_5^\gamma (x,p)$ be a solution of 
    \[
        \left\{ 
        \begin{array}{ll}
            \hat{p}_1 \partial_x J_5^\gamma - LJ_5^\gamma + \gamma P^+_0 \hat{p}_1 J_5^\gamma = 0,\quad & x\in \mathbb{R}^+,p\in\mathbb{R}^3\\
            J_5^\gamma(0,p) = j_5^\gamma(0,p),& p_1\ge0.
        \end{array} 
        \right.
    \]
    Thus, $k_5^\gamma$ satisfies the following equation
    \[
        \left\{ 
        \begin{array}{ll}
            \hat{p}_1 \partial_x k_5^\gamma - Lk_5^\gamma + \gamma P^+_0 \hat{p}_1 k_5^\gamma = \gamma e^{-\gamma x} P_1 \hat{p}_1 \Xi_5(p),\quad & x\in \mathbb{R}^+,p\in\mathbb{R}^3\\
            k_5^\gamma(0,p) = 0,& p_1 \ge 0.
        \end{array} 
        \right.
    \]
    Let $k_5^\gamma = e^{-\tau x} l_5^\gamma$, we have that
    \[
        \left\{ 
        \begin{array}{ll}
            \hat{p}_1 \partial_x l_5^\gamma - \tau \hat{p}_1 l_5^\gamma - Ll_5^\gamma + \gamma P^+_0 \hat{p}_1 l_5^\gamma = \gamma e^{-(\gamma - \tau ) x} P_1 \hat{p}_1 \Xi_5(p),\quad & x\in \mathbb{R}^+,p\in\mathbb{R}^3\\
            l_5^\gamma(0,p) = 0,& p_1\ge0.
        \end{array} 
        \right.
    \]
    Note that if we let $\zeta=\gamma e^{-(\gamma - \tau ) x} P_1 \hat{p}_1 \Xi_5(p)$, then the equation for $l_5^\gamma$ is  the one for $h$ in \eqref{eq:h:2}.
    By selecting the appropriate $\tau$ and $\gamma$ to satisfy that
    \[
        |\gamma|, |\tau| \ll 1,\quad \gamma = O(1) \tau, \quad \text{for some } O(1) > 2, 
    \]
    we obtain from the estimate in Theorem \ref{thm:lin:existence} that
    \begin{equation}\label{eq:est:J}
    \begin{aligned}
        \| P_0^+ \hat{p}_1 J_5^\gamma\| \mid_{x=0} 
        = & \| P_0^+ \hat{p}_1 (j_5^\gamma + k_5^\gamma)\| \mid_{x=0} \\ 
        \ge & \| P_0^+ \hat{p}_1 j_5^\gamma\| |_{x=0} - \|P_0^+ \hat{p}_1 l_5^\gamma\||_{x=0} \\ 
        \ge & \| P_0^+ \hat{p}_1 j_5^\gamma \| |_{x=0} - c \langle |\hat{p}_1|l_5^{\gamma,0}, l_5^{\gamma,0} \rangle_-^{1/2} \\
        \ge & \| P_0^+ \hat{p}_1 j_5^\gamma \||_{x=0} - C \gamma \ne 0. 
    \end{aligned}
    \end{equation}
    If we choose $a_0(p)=\Xi_5(p)$ for $p_1 \ge 0$, then \eqref{eq:est:J} implies that for sufficient small $\gamma$,
    \[
        P_0^+ \hat{p}_1 \mathbb{I}^\gamma (a_0) \ne 0.
    \]
    Therefore, $\langle \Xi_5,P_0^+ \hat{p}_1 \mathbb{I}^\gamma (a_0) \rangle$ defines a non-trivial bounded functional from $L_{\hat{p}_1,+}^2$ to $\mathbb{R}$. By Riesz representation theorem, there exists $r_5 \in L_{\hat{p}_1,+}^2$ such that
    \[
        \langle \Xi_5,P_0^+ \hat{p}_1 \mathbb{I}^\gamma (a_0) \rangle = \langle r_5,\hat{p}_1 a_0 \rangle_+.
    \]
    This shows, when $-1<\mathcal{M}^\infty<0$, that 
    \[
        \codim ( \{a_0 \in L_{\hat{p}_1,+}^2 \mid P_0^+ \hat{p}_1 \mathbb{I}^\gamma(a_0) = 0\})=1.
    \]

    When $0<\mathcal{M}^\infty<1$, $ \dim ( P_0^+ ) = 4 $. The range of $P_0^+$ is spanned by $\{ \Xi_2, \Xi_3, \Xi_4, \Xi_5\}$. Then $P_0^+ \hat{p}_1 \mathbb{I}^\gamma (a_0)$ can be written as
    \[
        P_0^+ \hat{p}_1 \mathbb{I}^\gamma (a_0) = \sum_{i=2}^{5}{\frac{\langle \Xi_i,P_0^+ \hat{p}_1 \mathbb{I}^\gamma (a_0) \rangle}{\langle \Xi_i,\Xi_i \rangle}\Xi_i}.
    \]
    Similar to the construction of $r_5$ in the case $-1<\mathcal{M}^\infty<0$, we can use $j_i^\gamma$, for $i=2,\cdots,5$, to find $r_i$ such that $\langle \Xi_i,P_0^+ \hat{p}_1 \mathbb{I}^\gamma (a_0) \rangle$ is a bounded functional from $L_{\hat{p}_1,+}^2$ to $\mathbb{R}$, and 
    \begin{equation}\label{eq:class:inn:product}
        \langle \Xi_i,P_0^+ \hat{p}_1 \mathbb{I}^\gamma (a_0) \rangle = \langle r_i,\hat{p}_1 a_0 \rangle_+, \qquad i=2,\cdots,5. 
    \end{equation}
    This shows, when $0<\mathcal{M}^\infty<1$, that 
    \[
        \codim ( \{a_0 \in L_{\hat{p}_1,+}^2 \mid P_0^+ \hat{p}_1 \mathbb{I}^\gamma(a_0) = 0\}) = 4.
    \]

    Similar to the discussion above, we can prove that when $\mathcal{M}^\infty>1$, the theorem still holds.
\end{proof}

\subsection{Classification of $P_0^+ \hat{p}_1 \mathbb{J}^\gamma(a_0) = 0$}

When $\mathcal{M}^\infty < -1$, $P_0^+ \equiv 0$. Therefore, we only need to consider the case when $\mathcal{M}^\infty > -1$. 

In order to classify $P_0^+ \hat{p}_1 \mathbb{J}^\gamma(a_0) = 0$ from the corresponding linear operator, we can consider the Frech\'{e}t derivative of $P_0^+ \hat{p}_1 \mathbb{J}^\gamma(a_0)$ and show that it is non-trivial in a space of dimension $n^+$. For this, we normalize the vectors $r_i$, $i=1,\cdots,5$, obtained for the linear operator $P_0^+ \hat{p}_1 \mathbb{I}^\gamma(a_0)$, such that
\[
    \langle r_i, \hat{p}_1 r_j \rangle_+ = \delta_{ij}, 
\]
if such $r_i$ exists. 

From the analysis in the above subsection, the boundary data $a_0$ can be decomposed as in Table \ref{tab_non_decom}. 

\begin{table}[H]\label{tab:lin}
    \renewcommand\arraystretch{1.25}
    \setlength{\abovecaptionskip}{-0.1cm}
    \setlength{\belowcaptionskip}{0.2cm}
        \centering
        \caption{Decomposition of $a_0$}
        \label{tab_non_decom}
        \begin{tabular}{lll}
        \toprule
             $-1<\mathcal{M}^\infty<0$  & $a_0=\langle r_5,\hat{p}_1a_0\rangle_+r_5+s$ & $\langle s,\hat{p}_1 r_5\rangle_+=0$   \\
             $0<\mathcal{M}^\infty<1$  & $a_0=\sum_{i=2}^{5}{\langle r_i,\hat{p}_1a_0\rangle_+r_i+s}$ & $\langle s,\hat{p}_1 r_i\rangle_+=0, i=2,\cdots,5 $   \\
             $\mathcal{M}^\infty>1$  & $a_0=\sum_{i=1}^{5}{\langle r_i,\hat{p}_1a_0\rangle_+r_i+s}$ & $\langle s,\hat{p}_1 r_i\rangle_+=0, i=1,\cdots,5 $\\   
        \bottomrule
        \end{tabular}
\end{table}

According to the decomposition in Table \ref{tab_non_decom}, we can denote 
\[
    a_0 := \sum_{i=6-n^+}^5 t_i r_i + s, \qquad \langle s, \hat{p}_1 r_i \rangle_+ = 0, \quad i = 6 - n^+, \cdots, 5,
\]
reparameterize $\mathbb{J}^\gamma(a_0)$ and define a new operator $ \bar{\mathbb{J}}^\gamma $ as in Table \ref{tab_non_para}. 

\begin{table}[H]
    \renewcommand\arraystretch{1.25}
    \setlength{\abovecaptionskip}{-0.1cm}
    \setlength{\belowcaptionskip}{0.2cm}
        \centering
        \caption{Reparameterization of $\mathbb{J}^\gamma$}
        \label{tab_non_para}
        \begin{tabular}{lll}
        \toprule
             $-1<\mathcal{M}^\infty<0$  & $\mathbb{J}^\gamma(a_0) =: \bar{\mathbb{J}}^\gamma(t_5,s)$ & $t_5\in\mathbb{R},s\in r_5^\bot$   \\
             $0<\mathcal{M}^\infty<1$  & $\mathbb{J}^\gamma(a_0) =: \bar{\mathbb{J}}^\gamma(t_2,\cdots,t_5,s)$ & $t_2,\cdots,t_5\in\mathbb{R},s\in \bigcap_{i=2}^{5}{r_i^\bot}$   \\
             $\mathcal{M}^\infty>1$  & $\mathbb{J}^\gamma(a_0) =: \bar{\mathbb{J}}^\gamma(t_1,\cdots,t_5,s)$ & $t_1,\cdots,t_5\in\mathbb{R},s\in \bigcap_{i=1}^{5}{r_i^\bot}$\\   
        \bottomrule
        \end{tabular}
\end{table}

We have the following lemma on the Frech\'{e}t derivative of $P_0^+ \hat{p}_1 \mathbb{J}^\gamma(a_0)$.

\begin{lemma}\label{lem:non:fre}
    Referring to Table \ref{tab_non_para}. There exists a small neighborhood $U \subset L_{\hat{p}_1,+}^2 \cap L_{p,\beta}^\infty(\mathbb{R}^3_+)$, $O\in U$, $\beta>2$, such that for any function $s\in U$, there exists a continuously differentiable function $a_0(s)$ such that $a_0 \in L_{\hat{p}_1,+}^2 \cap L_{p,\beta}^\infty(\mathbb{R}^3_+)$ and $P_0^+ \hat{p}_1 \mathbb{J}^\gamma(a_0)=0$. 
\end{lemma}
\begin{proof}
    It is obvious that $\mathbb{J}^\gamma(0)=0$, and 
    \[
        \frac{d}{d\epsilon}\mathbb{J}^\gamma(\epsilon a_0) \bigg|_{\epsilon=0}=\mathbb{I}^\gamma(a_0).
    \]
    From this, we have 
    \begin{equation}\label{eq:par:non}
        \frac{d}{d\epsilon}P_0^+ \hat{p}_1\mathbb{J}^\gamma(\epsilon a_0) \bigg|_{\epsilon=0} = P_0^+ \hat{p}_1\mathbb{I}^\gamma(a_0). 
    \end{equation}
    For $a_0$ suitably small, Theorem \ref{thm:lin:existence} and \ref{thm:non:exist} implies that $P_0^+ \hat{p}_1\mathbb{J}^\gamma$ is continuously differentiable on $a_0$. 
    Denote $\mathcal{W} := L_{\hat{p}_1,+}^2 \cap L_{p,\beta}^\infty(\mathbb{R}^3_+)$. 

    We consider the case of $0 < \mathcal{M}^\infty < 1$ for illustration. Then $n^+=4$, and $P_0^+ \hat{p}_1 \bar{\mathbb{J}}^\gamma$ is a map from $\mathbb{R}^{4} \times \mathcal{W}$ to $\mathbb{R}^{4}$ where
    \[
        P_0^+ \hat{p}_1 \bar{\mathbb{J}}^\gamma (t_2,t_3,t_4,t_5,s) 
        = P_0^+ \hat{p}_1 \mathbb{J}^\gamma (\sum_{i=2}^5 t_i r_i + s). 
    \] 

    Obviously, $\bar{\mathbb{J}}^\gamma(0) = 0$. From \eqref{eq:par:non}, we have 
    \[
        \frac{d}{d t_i} P_0^+ \hat{p}_1 \bar{\mathbb{J}}^\gamma \big|_{(0,0)} = P_0^+ \hat{p}_1 \mathbb{I}^\gamma (r_i), \qquad i=2,\cdots,5. 
    \]
    Then by \eqref{eq:class:inn:product}
    \[
        \frac{d}{d t_i} \langle \Xi_j, P_0^+ \hat{p}_1 \bar{\mathbb{J}}^\gamma \rangle \big|_{(0,0)} = \langle \Xi_j, P_0^+ \hat{p}_1 \mathbb{I}^\gamma (r_i) \rangle 
        = \langle r_j,\hat{p}_1 r_i \rangle_+ = \delta_{ij}. 
    \]
    Therefore, the Jacobian matrix 
    \[
        \frac{\partial (P_0^+ \hat{p}_1 \bar{\mathbb{J}}^\gamma)}{ \partial (t_2,t_3,t_4,t_5)}
    \]
    is invertible. 
    
    By the implicit function theorem, there exists a neighborhood $U$ of zero, such that for any $s\in U$, there exist a unique continuously differentiable function $t_i(s)$, $i=2,\cdots,5$, such that 
    \[
        P_0^+ \hat{p}_1 \bar{\mathbb{J}}^\gamma(t_2(s), t_3(s), t_4(s), t_5(s), s)\equiv 0. 
    \]
    Hence $a_0 = \sum_{i=2}^{5} t_i(s) r_i + s$ is the desired function of $s$. Note that the smallness of $s$ implies the smallness of $t_i$, as well as the smallness of $a_0$. This shows the theorem holds for $0 < \mathcal{M}^\infty < 1$. Similar to the discussion above, we can prove that when $-1 < \mathcal{M}^\infty <0$ or $\mathcal{M}^\infty>1$, the lemma still holds. 
\end{proof}
\begin{theorem}\label{thm:non:classi}
    The classification of boundary value near the far field Maxwellian state that satisfies the solvability condition \eqref{eq:sol:con:non} can be summarized as in Table \ref{tab_non}. 
    \begin{table}[H]\label{tab:non}
        \renewcommand\arraystretch{1.25}
        \setlength{\abovecaptionskip}{-0.1cm}
        \setlength{\belowcaptionskip}{0.2cm}
        \centering
        \caption{Classification of the solvability conditions for the nonlinear problem}
        \label{tab_non}
        \begin{tabular}{ll}
            \toprule
            $\mathcal{M}^\infty<-1$ \quad & $ \codim (\{a_0 \in L_{\hat{p}_1,+}^2\cap L_{p,\beta}^\infty(\mathbb{R}^3_+) \mid P_0^+ \hat{p}_1 \mathbb{J}^\gamma(a_0) = 0\})=0$   \\
            $-1<\mathcal{M}^\infty<0$  & $ \codim (\{a_0 \in L_{\hat{p}_1,+}^2 \cap L_{p,\beta}^\infty(\mathbb{R}^3_+)\mid P_0^+ \hat{p}_1 \mathbb{J}^\gamma(a_0) = 0\})=1$   \\
            $0<\mathcal{M}^\infty<1$  & $ \codim (\{a_0 \in L_{\hat{p}_1,+}^2 \cap L_{p,\beta}^\infty(\mathbb{R}^3_+) \mid P_0^+ \hat{p}_1 \mathbb{J}^\gamma(a_0) = 0\})=4$   \\
            $\mathcal{M}^\infty>1$  & $ \codim (\{a_0 \in L_{\hat{p}_1,+}^2 \cap L_{p,\beta}^\infty(\mathbb{R}^3_+) \mid P_0^+ \hat{p}_1 \mathbb{J}^\gamma(a_0) = 0\})=5$   \\
            \bottomrule
        \end{tabular}
    \end{table}
\end{theorem}
\begin{proof}
    The theorem can be derived from the codimensionality of the boundary conditions for the linear problem in Theorem \ref{lem:non:codi} and Lemma \ref{lem:non:fre}.
\end{proof}

Theorem \ref{thm:main:result} can be derived directly from Theorem \ref{thm:non:exist}, \ref{thm:persistance:init:cond}, and \ref{thm:non:classi}. 

\appendix

\titleformat{\section}{\Large\normalfont\bfseries}{}{0pt}{}
\renewcommand\thesubsection{\Alph{subsection}}
\counterwithin{lemma}{subsection}
\counterwithin{definition}{subsection}

\section{Appendix}

\subsection{Lorentz transformation}\label{Lorentz}

Here we introduce some computations and lemmas related to the Lorentz transformation for completion. 

\begin{definition}
We say that $\Lambda$ is a \emph{Lorentz transformation} if 
\[
    P \cdot Q = (\Lambda P) \cdot (\Lambda Q),
\]
holds for any two four-vectors $P$ and $Q$. 
Equivalently, $\Lambda$ is a Lorentz transformation if
\[
    \Lambda^T D \Lambda = D,
\]
where $D$ is given in \eqref{eq:D}
\end{definition}

\begin{lemma}[\cite{Strain:applications}]\label{lem:LT:mul}
    If $\Lambda$ is a Lorentz transformation, then $\Lambda^T$ and $\Lambda^{-1}$ are also Lorentz transformations. If $\Lambda_1$ and $\Lambda_2$ are Lorentz transformations, then $\Lambda_1 \Lambda_2$ is a Lorentz transformation.  
\end{lemma}

\begin{lemma}\label{lem:pbar}
    Suppose that $P$ and $Q$ are two four-vectors satisfying that $p_0^2 = |p|^2 + \mathfrak{c}^2$, $q_0^2 = |q|^2 + \mathfrak{c}^2$. 
    There exists a Lorentz transformation $\Lambda$, which is not unique, such that 
    \begin{equation}\label{eq:Lambda:P+Q}
        \Lambda(P+Q) = (\sqrt{s},0,0,0)^T, 
    \end{equation}
    where $s$ is defined by \eqref{eq:s}. Moreover, for any $\Lambda$ satisfying \eqref{eq:Lambda:P+Q}, it holds that $\Lambda(P-Q) = (0, \bar{p})^T$ for some $\bar{p}\in\mathbb{R}^3$ with $|\bar{p}| = 2g$, where $g$ is given by \eqref{eq:g}. 
\end{lemma}
\begin{proof}
    Let $\Lambda$ be a Lorentz transformation such that $\Lambda(P+Q) = (a,0,0,0)^T$ for some $a$, we have that
    \[
        a^2 = \Lambda(P+Q) \cdot \Lambda(P+Q) = (P+Q) \cdot (P+Q). 
    \]
    Therefore
    \[
        a = \sqrt{(P+Q) \cdot (P+Q)} = \sqrt{2 \mathfrak{c}^2 + 2 P\cdot Q} = \sqrt{s}. 
    \]
        
    On the other hand, 
    \[
        \Lambda P \cdot \Lambda (P+Q) = P \cdot (P+Q) = p_0 (p_0 + q_0) - p \cdot (p+q) = \mathfrak{c}^2 + P \cdot Q = \frac{1}{2} a^2,
    \]
    and 
    \[
        \Lambda P \cdot \Lambda (P+Q) = \Lambda P \cdot (a,0,0,0)^T = a (\Lambda P)^0. 
    \]
    Hence $(\Lambda P)^0 = \frac{a}{2}$. Therefore, we obtain that
    \[
        \Lambda (P-Q) = \Lambda (2P-(P+Q)) = 2 \Lambda P - (a,0,0,0)^T, 
    \]
    \[
        (\Lambda(P-Q))^0 = 2 (\Lambda P)^0 - a = 0. 
    \]
    Since we denote $\bar{p}$ through $\Lambda(P-Q) = (0, \bar{p})$, it yields by \eqref{eq:g} that 
    \[
        4 g^2 = - (P-Q) \cdot (P-Q) = - \Lambda(P-Q) \cdot \Lambda(P-Q) = |\bar{p}|^2. 
    \]
\end{proof}

\begin{lemma}\label{lem:post:tilde}
    Suppose that four-vectors $P,Q,P',Q'$ satisfy $P\cdot P = Q\cdot Q = P'\cdot P' = Q'\cdot Q' = \mathfrak{c}^2$ and $P+Q = P'+Q'$. 
    Then there exist a Lorentz transformation $\Lambda$ and $\omega\in\mathbb{S}^2$, such that $\Lambda(P+Q)=(\sqrt{s},0,0,0)$, $\Lambda(P-Q)=(0,\bar{p})$ and
    \[
        P' = \frac{1}{2} \Lambda^{-1} \begin{pmatrix} \sqrt{s} \\ 2 g \omega \end{pmatrix}, \qquad 
        Q' = \frac{1}{2} \Lambda^{-1} \begin{pmatrix} \sqrt{s} \\ -2 g \omega \end{pmatrix}. 
    \]
\end{lemma}
\begin{proof}
    By Lemma \ref{lem:pbar}, there exist a Lorentz transformation $\Lambda$ such that $\Lambda(P+Q)=(\sqrt{s},0,0,0)$ and $\Lambda(P-Q)=(0,\bar{p})$. Thus we have
    \[
        \Lambda (P'+Q') = \Lambda(P+Q)=(\sqrt{s},0,0,0). 
    \]
    From the proof of Lemma \ref{lem:pbar}, $\Lambda(P'-Q')=(0,\tilde{p})$ with $|\tilde{p}| = 2 g$. Denote $\tilde{p} = 2 g \omega$, $\omega\in\mathbb{S}^2$. This completes the proof of this lemma. 
\end{proof}

\begin{remark}\label{rmk:Lambda:U}
    For any four-vector $U$, there exists a Lorentz transformation, which is not necessarily unique and depends only on $U$, such that $\Lambda U = (\sqrt{U\cdot U}, 0, 0, 0)^T$. We denote such Lorentz transformation by $\Lambda_U$. 
\end{remark}

\begin{lemma}\label{lem:theta}
    Suppose that $P$ and $Q$ are two four-vectors satisfying that $p_0^2 = |p|^2 + \mathfrak{c}^2$, $q_0^2 = |q|^2 + \mathfrak{c}^2$. Let $\hat{\Lambda}$ be any Lorentz transformation, and $\omega\in\mathbb{S}^2$. Then $\Lambda_{\hat{\Lambda} (P+Q)} \hat{\Lambda} \Lambda_{P+Q}^{-1} (0,\omega)^T = (0,\tilde{\omega})^T$ for some $\tilde{\omega} \in \mathbb{S}^2$. 
\end{lemma}
\begin{proof}
    By definition of $\Lambda_{P+Q}$, we have that 
    \[
        \Lambda_{P+Q} (P+Q) = (\sqrt{(P+Q) \cdot (P+Q)}, 0, 0, 0)^T.  
    \]
    Hence, 
    \begin{align*}
        \Lambda_{\hat{\Lambda} (P+Q)} \hat{\Lambda} \Lambda_{P+Q}^{-1} (\sqrt{(P+Q) \cdot (P+Q)}, 0, 0, 0)^T = & \Lambda_{\hat{\Lambda} (P+Q)} \hat{\Lambda} (P+Q) \\
        = & (\sqrt{[\hat{\Lambda} (P+Q)]\cdot [\hat{\Lambda} (P+Q)]}, 0, 0, 0)^T \\
        = & (\sqrt{(P+Q) \cdot (P+Q)}, 0, 0, 0)^T. 
    \end{align*}
    Denote $M:=\Lambda_{\hat{\Lambda} (P+Q)} \hat{\Lambda} \Lambda_{P+Q}^{-1} = (M^i_j)_{i,j=0,1,2,3}$. It is obvious that $M$ is a Lorentz transformation. Here, $M^i (i=0,1,2,3)$ are used to denote the rows, and $M_j (j=0,1,2,3)$ are used to denote the columns of the matrix $M$. By linearity, we know that $M (1,0,0,0)^T = (1,0,0,0)^T$. Therefore, $M_0 = (1,0,0,0)^T$. From the orthogonal properties of the Lorentz transformation $M_i \cdot M_j =D_{ij}$, $i,j=0,1,2,3$, we know that $M^0 = (1,0,0,0)$, and
    \[
        M_{00} := 
            \begin{pmatrix}
                M^1_1 & M^1_2 & M^1_3 \\
                M^2_1 & M^2_2 & M^2_3 \\
                M^3_1 & M^3_2 & M^3_3 \\
            \end{pmatrix}
    \]
    is an orthogonal matrix. Therefore, $M (0,\omega)^T = (0,\tilde{\omega})^T$ for some $\tilde{\omega}\in \mathbb{S}^2$. 
\end{proof}

\begin{lemma}[\cite{Strain:applications}]\label{lem:Lorentz:inv}
    Suppose that a four-vector $P$ satisfies $p_0^2 = |p|^2 + \mathfrak{c}^2$, and $P' = \Lambda P$ for some Lorentz transformation $\Lambda$. Then $\frac{dp}{p_0} = \frac{dp'}{p'_0}$, namely, $\frac{dp}{p_0}$ is a Lorentz invariant measure. 
\end{lemma}

\begin{lemma}\label{lem:tildepq:equiv}
    Suppose that $P$ and $Q$ are four-vectors satisfying that $P \cdot P = Q \cdot Q = \mathfrak{c}^2$. Let $\Lambda$ be any Lorentz transformation, $\tilde{P} := \Lambda P$ and  $\tilde{Q} := \Lambda Q$. Then
    
    {\rm (i)} $|\tilde{p} - \tilde{q}| \sim |p - q|$, 

    {\rm (ii)} $\tilde{p}_0 \sim p_0 $, $\tilde{q}_0 \sim q_0$,

    Here the notation "$A \sim B$" means that there exist a constant $c>0$, depending only on $\Lambda$, such that $ A/c \le B \le c A$. 
\end{lemma}
\begin{proof}
    (i) Denote $\Lambda$ as
    \[
        \Lambda = \begin{pmatrix}
            \lambda_{00} & b^T \\
            a & \tilde{\Lambda}
        \end{pmatrix}. 
    \]
    Without loss of generality, we assume $\lambda_{00}>0$. Since $\Lambda^{-1}$ is also a Lorentz transformation, 
    the four-vectors $\tilde{P}$ and $\tilde{Q}$ are equivalent to $P$ and $Q$. Obviously,
    $\tilde{p} = p_0 a + \tilde{\Lambda} p$, $\tilde{q} = q_0 a + \tilde{\Lambda} q$, and 
    \begin{equation}\label{eq:p-q:prime}
        |\tilde{p} - \tilde{q}| \le |p_0-q_0| a + |\tilde{\Lambda} (p-q)|. 
    \end{equation}

    Since 
    \[
        |p_0 - q_0| = |\sqrt{\mathfrak{c}^2 + |p|^2} - \sqrt{\mathfrak{c}^2 + |q|^2} | = \frac{ | p\cdot p - q \cdot q | }{\sqrt{\mathfrak{c}^2 + |p|^2} + \sqrt{\mathfrak{c}^2 + |q|^2}} \le | p - q |, 
    \]
    we have by \eqref{eq:p-q:prime} that 
    \[
        |\tilde{p} - \tilde{q}| \le c |p-q|,
    \]
    for some constant $c$ depending only on $\Lambda$. Note that $P,Q$ and $\tilde{P},\tilde{Q}$ are equivalent, hence $|\tilde{p} - \tilde{q}| \sim |p-q|$. 

    (ii) Note that $\tilde{p}_0 = \lambda_{00} p_0 + b \cdot p$, $\tilde{q}_0 = \lambda_{00} q_0 + b \cdot q$ and $\Lambda^i \cdot \Lambda^j = D_{ij}, (i,j=0,1,2,3)$, where $\Lambda^i$ denotes the rows of $\Lambda$  and $D_{ij}$ is given in \eqref{eq:D}. Therefore, we have that $\lambda_{00}^2 = b \cdot b + 1$ and 
    \[
        \tilde{p}_0 \le 2 \lambda_{00} p_0, \quad \tilde{q}_0 \le 2 \lambda_{00} q_0.
    \]
    On the other hand, we have that
    \begin{align*}
        \tilde{p}_0 & = \lambda_{00} p_0 + b \cdot p\\
        & = \lambda_{00} p_0 - |b| p_0 + |b| p_0 + b \cdot p\\
        & \ge (\lambda_{00} - |b|) p_0 + |b| p_0 - |b||p|\\
        & \ge \frac{\lambda_{00}^2-|b|^2}{\lambda_{00}+|b|} p_0 \ge \frac{1}{2 \lambda_{00}} p_0, 
    \end{align*}
    Thus 
    \begin{equation}\label{eq:tildep:equiv}
        \frac{1}{2 \lambda_{00}} p_0 \le \tilde{p}_0 \le 2 \lambda_{00} p_0,
    \end{equation}
    which implies that $\tilde{p}_0 \sim p_0$.
    Similarly, we have $\tilde{q}_0 \sim q_0$.
\end{proof}

\subsection{Moments of the relativistic Maxwellian}

The following lemma is well-known and can be found in some literature, see for example \cite{Bellouquid:Calvo:Nieto:Soler,Yang:Yu:Euler}. We present here a general version with light speed. 

\begin{lemma}\label{lem:moments}
    Let $J = J_{[1, u,T]}$, $\displaystyle z:=\frac{\mathfrak{c}^2}{T}$. Then
    \begin{enumerate}
        \item $ \displaystyle \int_{\mathbb{R}^3}\frac{J}{p_0}dp=\frac{K_1(z)}{ \mathfrak{c} K_2(z)}$, 
            
        \item $ \displaystyle \int_{\mathbb{R}^3} \frac{p_i}{p_0} J dp= \frac{u_i}{\mathfrak{c}},\quad i=0,1,2,3$, 
            
        \item $\displaystyle \int_{\mathbb{R}^3}p_0 Jdp= \frac{u_0^2}{\mathfrak{c}} (  \frac{K_1(z)}{K_2(z)} + \frac{3}{z} ) + \frac{1}{z} \frac{|u|^2}{\mathfrak{c}} = \frac{u_0^2}{\mathfrak{c}} \frac{K_3(z)}{K_2(z)} - \frac{\mathfrak{c}}{z} $, 
            
        \item $ \displaystyle \int_{\mathbb{R}^3}\frac{p_i^2}{p_0} J dp = \frac{u_i^2}{\mathfrak{c}}( \frac{K_1(z)}{K_2(z)} + \frac{4}{z} )+\frac{\mathfrak{c}}{z}=\frac{u_i^2}{\mathfrak{c}}\frac{K_3(z)}{K_2(z)}+\frac{\mathfrak{c}}{z},\mathbb \quad i=1,2,3$, 
            
        \item $ \displaystyle \int_{\mathbb{R}^3}p_i Jdp=\frac{u_0 u_i}{\mathfrak{c}}(  \frac{K_1(z)}{K_2(z)}  + \frac{4}{z} )=\frac{u_0u_i}{\mathfrak{c}}\frac{K_3(z)}{K_2(z)}, \quad i=1,2,3$, 
        
        \item $ \displaystyle \int_{\mathbb{R}^3} \frac{p_i p_j}{p_0} Jdp = \frac{u_i u_j}{\mathfrak{c}}(  \frac{K_1(z)}{K_2(z)} + \frac{4}{z} )=\frac{u_iu_j}{\mathfrak{c}}\frac{K_3(z)}{K_2(z)},\quad 0\le i \ne j\le 3$, 
        
        \item $ \displaystyle \int_{\mathbb{R}^3} p_0^2 J dp = \frac{u_0}{\mathfrak{c}}(u_0^2+\frac{1}{z}(3u_0^2+3 |u|^2)\frac{K_3(z)}{K_2(z)})$, 
        
        \item $ \displaystyle \int_{\mathbb{R}^3} p_i^3 J \frac{dp}{p_0} = \frac{u_i}{\mathfrak{c}}(u_i^2+\frac{1}{z}(6u_i^2+3c^2)\frac{K_3(z)}{K_2(z)}), \quad i=1,2,3$, 
        
        \item $ \displaystyle \int_{\mathbb{R}^3} p_0 p_i J dp = \frac{u_i}{\mathfrak{c}}(u_0^2+\frac{1}{z}(6u_0^2-\mathfrak{c}^2)\frac{K_3(z)}{K_2(z)}),\quad i=1,2,3$, 
        
        \item $ \displaystyle \int_{\mathbb{R}^3}  p_i^2 J dp = \frac{u_0}{\mathfrak{c}}(u_i^2+\frac{1}{z}(6u_i^2+\mathfrak{c}^2)\frac{K_3(z)}{K_2(z)}),\quad i=1,2,3$, 
        
        \item $ \displaystyle \int_{\mathbb{R}^3}  \frac{p_i^2 p_j}{p_0} J dp = \frac{u_j}{\mathfrak{c}}(u_i^2+\frac{1}{z}(6u_i^2+\mathfrak{c}^2)\frac{K_3(z)}{K_2(z)}),\quad  1\le i\ne j \le 3$, 
        
        \item $ \displaystyle \int_{\mathbb{R}^3}  p_i p_j J dp = \frac{u_0u_iu_j}{\mathfrak{c}^3}(\mathfrak{c}^2+6T\frac{K_3(z)}{K_2(z)}),\quad 1\le i\ne j\le 3$, 
        
        \item $ \displaystyle \int_{\mathbb{R}^3}  \frac{p_1 p_2 p_3}{p_0} J dp = \frac{u_1u_2u_3}{\mathfrak{c}^3}(\mathfrak{c}^2+6T\frac{K_3(z)}{K_2(z)})$, 
    \end{enumerate}
    where $K_\alpha(z)$ is the modified Bessel function defined by
    \[
        K_\alpha(z) := \int_{0}^{\infty} e^{- z \cosh t }\cosh \alpha t dt, \qquad \mathrm{Re} z>0,\qquad \alpha\in\mathbb{R}.  
    \]
\end{lemma}
The lemma can be proved by using Lorentz transformation, the definition of the modified Bessel function, and the following lemma. We omit details for brevity. 
\begin{lemma}[\cite{Abramowitz:Stegun}]
    For any $\alpha\in\mathbb{R}$, it holds that
    \[
        K_{\alpha + 1}(z) = \frac{2\alpha}{z} K_\alpha(z) + K_{\alpha - 1}(z). 
    \]
\end{lemma}

\bibliographystyle{plain}
\bibliography{refs.bib}

\end{document}